\documentclass{amsart}
\usepackage{mathtools}
\usepackage{amsmath,bm}
\newcommand{\tend}[1]{\hbox{\oalign{$\bm{#1}$\crcr\hidewidth$\scriptscriptstyle\bm{\sim}$\hidewidth}}}
\numberwithin{equation}{section}
\newcommand\norm[1]{\left\lVert#1\right\rVert}
\usepackage{stmaryrd}
\usepackage{tikz}
\usepackage{mathrsfs}
\usepackage{hyperref}
\usetikzlibrary{calc,angles,quotes}
\newtheorem{theorem}{Theorem}[section]
\newtheorem{remark}[theorem]{Remark}
\usepackage{xcolor}
\usepackage[normalem]{ulem}
\usepackage{cancel}

\newtheorem{lemma}[theorem]{Lemma}
\newtheorem{definition}[theorem]{Definition}
\usepackage{bm}
\newcommand{\ou}{\color{blue}}

{\let\oldqedsymbol=\qedsymbol
	\renewcommand{\qedsymbol}{$\blacktriangleleft$}
	\begin{proof}[\bfseries\upshape Solution]}
	{\end{proof}
	\renewcommand{\qedsymbol}{\oldqedsymbol}}
\colorlet{colorYO}{red}
\colorlet{colorSJ}{blue}

\newcommand{\SJ}[1]{\textcolor{colorSJ}{#1}}   

\title[Homogenization for a Variational Problem ]{Homogenization for a Variational Problem with a
Slip Interface Condition}

\author{Miao-jung Yvonne Ou}
\address{Department of Mathematical Sciences, University of Delaware,
408 Ewing Hall, Newark, Delaware 19716 USA}
\email{mou@udel.edu}

\author{Silvia Jim\'enez Bola\~nos}
\address{Department of Mathematics, Colgate University,
13 Oak Drive, Hamilton, NY 13346 USA}
\email{sjimenez@colgate.edu}

\date{\today}

\begin{document}
\maketitle
\begin{abstract}
Inspired by applications, we study the effect of interface slip on the effective wave propagation in poroelastic composites. The current literature on the homogenization for the poroelastic wave equations are all based on the no-slip interface condition posed on the micro-scale. However, for certain pore fluids, the no-slip conditions are known to be physically invalid. Even though there are results in a few papers regarding porous media with slip condition on the interface, they are for porous media with rigid solid matrix rather than an elastic one. For the former case, the equations for the micro-scale are posed only in the pore space and the slip on the interface involves only the fluid velocity and the fluid stress. For the latter case, both the fluid equations and the elastic equations are posed in the respective phases and the slip conditions involve the velocities on both sides of the interface, rather than just the fluid side. With this slip condition, a variational boundary value problem governing the small vibrations of a periodic mixture of an elastic solid and a slightly viscous fluid is studied in the paper.  The method of two-scale convergence is used to obtain the macroscopic behavior of the solution and to identify the role played by the slip interface condition. 
\end{abstract}

\maketitle

\section{Introduction}

Poroelastic materials are composite materials made of elastic solid matrix and fluid residing in the pore space, e.g. cancellous bones, saturated rocks and sea ice. To study the physical properties of these composite materials, the availability of the poroelastic wave equations for wavelength much larger than the scale of the micorstructure is crucial. In this wavelength regime, techniques such as the homogenization method can be used to derive these effective wave equations from the wave equations for each phase in the micro-scale. Compared with the effective media approach, the homogenization approach is less phenomenological in the sense that the coefficients in the homogenized equations can be calculated by solving the so-called cell problems, which are derived as part of the homogenization process. The homogenization for the variational boundary value problem of the stiff type that  governs the small vibrations of a periodic mixture of an elastic solid and a slightly viscous fluid, with no discontinuity of the displacement in the interface between the two phases, was developed by Nguetseng in \cite{Nguetseng1990}, where the resulting homogenized equations are the poroelastic wave equations for composites with no-slip interface conditions. Also, this set of equations validates the well-known Biot equations \cite{biot1956theory-of-propa,biot1956theory-high}. However, it has been observed that the no-slip interface condition are not valid for some applications such as the polymeric pore fluid or coated interface; see \cite{Sochi2011Slip-at-fluid-s} and the references therein. In these cases, the interface condition at the micro-scale is of slip type and it leads to a set of interesting questions. For example, since the no-slip condition is linked to the concept of the boundary layers within which the energy dissipation is the most significant, how will the energy dissipation change when the no-slip condition is replaced by a slip condition on the interface? In the homogenized equations, the energy dissipation and wave dispersion are described by the effective properties called 'dynamic permeability' and 'dynamic tortuosity', the two most important characterizations of the dynamic properties of the poroelastic materials. How will these quantities change when the no-slip conditions are replaced by a slip condition? As a starting point for answering these important questions, we carry out in this paper the analysis for the case in which a slip boundary condition at the  solid-fluid interface is allowed.  

	We consider the mixture of an elastic solid and a slightly viscous fluid, in the framework of small motions linearized with respect to a rest state, where the geometric distribution of the solid and fluid parts is periodic, with characteristic length of the period given by $\epsilon$, with $0<\epsilon<<1$.  Mixtures in mechanics are of great interest in physical applications, see for example, \cite{LEVY1979,nguetseng1980,SanchezPalencia1980,Nguetseng1990,Allaire91,Collin2018}.  

A variety of different problems arise according to the orders of the viscosity coefficients and the topological properties of the mixture.  In \cite{SH1980}, the authors used the energy method (see \cite{Bensoussan1978}) to show that, whether or not the fluid phase is connected, if the elasticity coefficients together with the viscosity coefficients are $O(\epsilon^0)$, the limit of the displacement, as $\epsilon\rightarrow 0$, does not depend on the local variables.  In this paper, we will take the elasticity coefficients  to be $O(\epsilon^0)$, and the viscosity coefficients to be $O(\epsilon^2)$ i.e. $\mu\epsilon^2$ and $\eta\epsilon^2$ with constant $\mu$ and $\eta$.  In the formal analysis, seen in \cite{LEVY1979} or Chapter~8 of {\cite{SanchezPalencia1980}}, it is concluded that, if the fluid part is strictly contained in the period of reference, and therefore it is not connected, the formal limit of the displacement in the mixture does not depend on the local variables.  On the other hand, if the fluid part intersects each face of the period of reference, and it is connected, the formal limit of the displacement depends on the local variables. The formal analysis results above were rigorously proved in \cite{Nguetseng1990} using the method of 2-scale convergence (see \cite{Nguetseng1989,Allaire1992}).  Differing from \cite{Nguetseng1990}, connectedness doesn't play a role in the analysis developed and the results obtained in this paper.

	{The novelty of this paper is that the results obtained in \cite{Nguetseng1990} are generalized to the case in which there is a slip interface condition.  Though the results of the present paper are similar to those of  \cite{Nguetseng1990}, dealing with the interface term (\ref{interfacecond}) is not trivial.}  New technical lemmas are required in order to carry out the limiting process.

	This paper is organized as follows.  In Section~\ref{background}, we present the set up of the mathematical problem.  In Section~\ref{VarForDer}, we derive (\ref{TVF5}), the variational formulation of the boundary value problem that governs the small vibrations of a solid-fluid mixture with a slip boundary condition on their interface .  In Section~\ref{EUProof}, we prove the existence and uniqueness of the solution to our problem for a fixed $\epsilon$.  The main general convergence and extension results can be found in Section~\ref{CVResult}. In Section~\ref{sec:premres}, we prove the necessary uniform estimates to find the macroscopic equation.  The derivation for the local problems (for $u_1(\mathbf{x},\mathbf{y})$ and $u_r(\mathbf{x},\mathbf{y})$ in Lemma~\ref{lem:due2scsol}) is done in Section~\ref{sec:derivlocprob}. In Section~\ref{HomogProbConvTh}, we derive the homogenized problem.  Finally, in Section~\ref{conclusion}, we present our conclusions.

\section{Background}
\label{background}
	{In this section, we state the mathematical formulation of the problem to be studied, which concerns the acoustics equations of poroelastic materials with periodic microstructure and a slip boundary condition on the solid-fluid interface. }  
\subsection{{Geometry of the microstructure}}	
	We consider the space $\mathbb{R}^3$ of the variables $\mathbf{y}=(y_1,y_2,y_3)$ to be a periodic set, with unit cell $Y=\left(-\frac{1}{2},\frac{1}{2}\right)^3$, decomposed as: 
\begin{equation*}
Y=Y_s\cup Y_f \cup \Gamma,
\end{equation*}
where $Y_s$ and $Y_f$ are open sets in $\mathbb{R}^3$, where $Y_s$ represents the part of $Y$ occupied by the solid and $Y_f$ represents the part of $Y$ occupied by the fluid, and $\Gamma$ is the smooth surface separating them. {The boundary, the closure and the Lebesgue measure of a measurable set $A$ in $\mathbb{R}^3$ are denoted by $\partial A$, $\overline{A}$ and $|A|$, respectively.} Let $\tilde{Y_{{I}}}$ be the $Y$-periodic extension of  $Y_{{I}}$ , i.e. the union of all the $(Y_{{I}}{\cup(\overline{Y_{{I}}}\cap\partial Y)})+\mathbf{k}$, $\mathbf{k}$ ranging over $\mathbb{Z}^3$, $I=s,f$.  Similarly, we denote by $\tilde{\Gamma}$ the $Y$-periodic extension of $\Gamma$.

		
		Following \cite{Allaire1989}, we assume the following hypotheses: 
		\begin{itemize}
		    \item[(i)]$Y_s$ and $Y_f$ have strictly positive measures on $\overline{Y}$. 
		    \item[(ii)] $\tilde{Y_f}$ and $\tilde{Y_s}$ are open sets with boundary of class $C^1$, and are locally located on one side of their boundary.  Moreover, $\tilde{Y_s}$ is connected. Hence $\overline{Y_s}$ has an intersection with each face of the cube $\overline{Y}$ with strictly positive surface measure. 
		    \item[(iii)] $Y_s$ is an open connected set with a locally Lipschitz boundary.
		\end{itemize}



\subsection{{Notation}}
{Let $\Omega$ be the smooth bounded open set occupied by the poroelastic material in space $\mathbb{R}^3$ with coordinates $\mathbf{x}=(x_1,x_2,x_3)$. Let $\epsilon$ denote the scale of the periodic microstructure, $0<\epsilon\ll1$. }

The solid part and the fluid part of $\Omega$, together with their interface, are defined as follows. 
\begin{equation*}
	\Omega_\epsilon^s=\Omega\cap\epsilon\tilde{Y_s},\hspace{1cm}\Omega_\epsilon^f=\Omega\cap\epsilon\tilde{Y_f},\hspace{1cm} \Gamma_\epsilon=\left\{\mathbf{x}\in\Omega\,:\,\frac{\mathbf{x}}{\epsilon}\in\tilde{\Gamma}\right\},
\end{equation*}

Observe that $\Omega_\epsilon^s$ is connected. {Since $\Gamma_\epsilon$ is orientable}, we can define $\Gamma_\epsilon^s$ and $\Gamma_\epsilon^f$ to be the solid side and the fluid side of $\Gamma_\epsilon$, respectively. {With this notation, the following jump operator across $\Gamma_\epsilon$ is defined
\begin{equation}
\llbracket\cdot\rrbracket_s^f:=(\cdot)\Big|_{\Gamma_\epsilon^f}-(\cdot)\Big|_{\Gamma_\epsilon^s}.
\label{jump_op}
\end{equation}
Moreover, we let $\mathbf{n}$ to be the unit outward normal vector of $\partial \Omega_\epsilon^f$, i.e. $\mathbf{n}$ points toward the solid phase.}

We denote by $d\sigma(\mathbf{y})$, for $\mathbf{y}\in Y$, and by $d\sigma_\epsilon(\mathbf{x})$, for $\mathbf{x}\in\Omega$, the surface measures on $\Gamma$ and $\Gamma_\epsilon$, respectively. Note that  $$\partial\Omega_\epsilon^s=\Gamma_\epsilon\cup\left(\partial\Omega_\epsilon^s\cap\partial\Omega\right)\hspace{3mm} \text{ and }\hspace{3mm} \Gamma_\epsilon=\partial\Omega_\epsilon^s\setminus(\partial\Omega_\epsilon^s\cap\partial\Omega).$$

{The 'micro' coordinates $\mathbf{y}$ and the 'macro'-coordinates $\mathbf{x}$ are related by $\mathbf{y}=\epsilon \mathbf{x}$. Also, the superscript $\epsilon$ is reserved for signifying the following re-scaling of a function $w$}  
\begin{equation*}
	w^\epsilon(\mathbf{x})=w\left(\frac{\mathbf{x}}{\epsilon}\right), \hspace{2mm}\text{ for $w=w(\mathbf{y})$}
\end{equation*}
 
 The gradient of a vector field $\mathbf{v}(\mathbf{x})$ is denoted by $\nabla \mathbf{v}$, which is a matrix such that  $[\nabla \mathbf{v}]_{ij}=\frac{\partial v_i}{\partial x_j}$. The linear strain tensor  with respect to $\mathbf{x}$ (resp. $\mathbf{y}$) is denoted by $\mathbf{E}(\cdot)$(resp. $\mathbf{e}(\cdot)$) is defined as follows:

	\begin{equation*}
		E_{ij}(\mathbf{v(\mathbf{{x}})})=\frac{1}{2}\left(\frac{\partial v^i}{\partial x_j}+\frac{\partial v^j}{\partial x_i}\right),\, i,j=1,2,3;
	\end{equation*}
	\begin{equation*}
		e_{ij}(\mathbf{w}(\mathbf{{y}}))=\frac{1}{2}\left(\frac{\partial w^i}{\partial y_j}+\frac{\partial w^j}{\partial y_i}\right),\,i,j=1,2,3.
	\end{equation*}
  We denote by ${\rm div}_\mathbf{y}$ the divergence operator with respect to $\mathbf{y}$, and by ${\rm div}_\mathbf{x}$, or simply ${\rm div}$, the same operator with respect to $\mathbf{x}$.  If $V$ is a vector space, the vector space of the same name written in boldface $\mathbf{V}$ represents the corresponding product space $V^3=V\times V\times V$.  In this paper, all the vector spaces considered are over the complex field $\mathbb{C}$.  The Einstein summation convention is used throughout the rest of the paper, $\delta_{ij}$ is the Kronecker delta, and $C$ represents a universal constant, which is independent of variable quantities such as $\epsilon$, $t$, ..., and that may change value from line to line.
\subsection{{Governing equations}}
	We denote the elastic moduli of the solid phase by constants $a_{ijkl}$, $1\leq i,j,k,l \leq3$, satisfying the following symmetry conditions and $V$-ellipticity condition
\begin{align}
	&a_{ijkl}=a_{jikl}=a_{ijlk}=a_{klij},\label{symmetry}\\
	&a_{ijkl}\xi_{kl}\xi_{ij}\geq c\,\xi_{ij}\xi_{ij}, \hspace{2mm} c>0, \hspace{1mm} \forall \xi_{ij}=\xi_{ji},	\label{a-coercive}
\end{align}	
	
	{For the fluid part, let $\eta\epsilon^2,\mu\epsilon^2\in \mathbb{R}$ be the fluid viscosities, where $\mu$ and $\eta$ are of $O(1)$ and satisfy the following conditions:}
\begin{align}
	\mu>0, \hspace{2mm} \frac{\eta}{\mu} > -\frac{2}{3}.\label{etamu} 
\end{align}
	{We} assume the {external} force $\mathbf{f}=\left\{f^i\right\}\in L^2_{loc}(0,+\infty; \mathbf{L}^2(\Omega))$ is independent of $\epsilon$ and satisfies the following bound: 
\begin{equation}
	\label{f}
	\norm{\mathbf{f}(t)}^2_{\mathbf{L}^2(\Omega)}\leq K e^{mt} \hspace{2mm} (K>0, m\in\mathbb{R}), \hspace{2mm}\text{ for almost all $0<t<\infty$.}
\end{equation} 

{{Let $\rho^s$ and $\rho^f$ be the density of the solid phase and the fluid phase, respectively, and $c_0$ the reference speed of sound}. For a fixed $\epsilon$, the governing equations for the solid phase are in terms of the} displacement {field} $\mathbf{u}_\epsilon$  \cite{SanchezPalencia1980}(Chapter 8):
\begin{align}
	\rho^s\frac{\partial^2 u_\epsilon^i}{\partial t^2}&=\frac{\partial \sigma^s_{ij}}{\partial x_j}+f^i,	\label{S1} \\
	\tend{\sigma}^s_{ij}&=a_{ijkl}E_{kl}(\mathbf{u}_\epsilon);	\label{S2}
	\end{align}
	
{whereas} for the fluid phase, they are given in terms of the fluid velocity field { $\frac{\partial \mathbf{u}_\epsilon}{\partial t}$ and the acoustic pressure $p_\epsilon:=-c_0^2\,\rho^f\,\nabla\cdot \mathbf{u}_\epsilon$  } \cite{SanchezPalencia1980}(Chapter 8):
	\begin{align}
	\rho^f\frac{\partial^2 u^i_\epsilon}{\partial t^2}&=\frac{\partial \sigma^f_{ij}}{\partial x_j}+f^i,		\label{F1}\\
 \tend{\sigma}^f_{ij}&=-\delta_{ij}p_\epsilon+\left(\epsilon^2\eta\delta_{ij}\delta_{kl}+2\mu\epsilon^2\delta_{ik}\delta_{jl}\right)E_{kl}\left(\frac{\partial \mathbf{u}_\epsilon}{\partial t}\right),		\label{F2}
	\end{align}
{Note that} the constitutive equation (\ref{F2}) {is equivalent with}:
	\begin{equation}
		\label{F8} 
		\tend{\sigma}^f_{ij}=\delta_{ij}c_0^2\rho^f\nabla\cdot \mathbf{u}_\epsilon+\left(\epsilon^2\eta\delta_{ij}\delta_{kl}+2\mu\epsilon^2\delta_{ik}\delta_{jl}\right)E_{kl}\left(\frac{\partial \mathbf{u}_\epsilon}{\partial t}\right).
	\end{equation}

{The equations of motion \eqref{S1}-\eqref{F2} are complemented by the jump conditions} on the interface ${\Gamma_\epsilon}$:
\begin{align}
	\tend{\sigma}^s{\ou \cdot}\mathbf{n}&=\tend{\sigma}^f  {\ou \cdot} \mathbf{n}  \mbox{    (continuity of stress)},\label{contnormstress}\\
	\tend{\sigma}^f\cdot \mathbf{n}&=\epsilon\,\alpha\left\llbracket\frac{\partial \mathbf{u_\epsilon}}{\partial t}\right\rrbracket_s^f  \mbox{   (slip condition),}
	\label{interfacecond}
\end{align} 
where $\alpha>0$ is the slip constant.  For more information about the $\epsilon$-scaling for the slip constant in (\ref{interfacecond}), see \cite{Allaire91,Cioranescu1996}. In this paper, we consider the homogeneous initial conditions:
\begin{equation}
\mathbf{u}_\epsilon(0)=0, \, \mathbf{v}_\epsilon(0)=\displaystyle\frac{\partial \mathbf{u}_\epsilon}{\partial t}(0)=0.
\label{ICs}
\end{equation}

\section{Variational Formulation}
\label{VarForDer}

{In this section, the variational formulation of the system of equations \eqref{S1}-\eqref{ICs} is derived. We start with introducing the function spaces which are used in this paper.}
\subsection{{Function spaces}}
\label{fun_spaces}
{ Recall that $Y=(-\frac{1}{2},\frac{1}{2})^3$. Let $Y_o$ be a subset of $Y$ with Lipschitz boundary such that the periodic extension $\tilde{Y_o}$ has $C^1$ boundary in $\mathbb{R}^3$. The following typical function spaces are used in this paper.}
\begin{itemize}
	\item $C_p$: the space of $Y$-periodic continuous functions on $\mathbb{R}^3$.
	
	\item $C_p^\infty$: the space of $Y$-periodic $C^\infty$ functions on $\mathbb{R}^3$. 

	\item $L^2_p(Y_o)$: the space of {$Y$-periodic square integrable function in $Y_o$}. This is a Hilbert space with the $L^2(Y_o)$-norm.
	
	\item $H^1_p(Y_o)$= $\left\{w|\,w\in L^2_{p}(Y_o),\frac{\partial w}{\partial y_i}\in L_p^2(Y_o),\, i=1,2,3\right\}$.  This is a Hilbert space with the $H^1(Y_o)$-norm.
	
	\item $H^1_p(Y_o)/\mathbb{C}^3=\left\{ w|\, w\in H^1_p(Y_o) \mbox{ and }\int_{Y_o}w\,d\mathbf{y}=0\right\},$ equipped with the norm $ \norm{w}^2_{H^1_p(Y_o)/\mathbb{C}}=\sqrt{\sum_{i=1}^N\norm{\frac{\partial w}{\partial y_i}}^2_{L^2(Y_o)}}$.
	
	\item $\mathscr{K}({\Omega})$: {The space of continuous functions} with compact support {in $\Omega$}.
	
	\item$\mathscr{D}(\Omega)$: space of $C^\infty$ functions with compact support in $\Omega$. 
\end{itemize}

In the case $Y_o=Y$, we will write $L^2_p$ (respectively, $H^1_p$) instead of $L^2_p(Y)$ (respectively, $H_p^1(Y)$).
\vspace{0.1in}

{The following function spaces specializing to the interface slip conditions \eqref{slipcondit} are considered in this paper.}

\begin{equation}
	\label{V}
 V:=\left\{\mathbf{w}\,|\,\mathbf{w}\in H^1(\Omega_\epsilon^f\cup \Omega_\epsilon^s)\, : \,\mathbf{w}\big|_{\partial\Omega}=\mathbf{0} \right\},
\end{equation}
with the norm given by: 
\begin{equation}
\label{Vnorm}
\norm{\mathbf{w}}_V^2:=\norm{\mathbf{w}}_{H^1(\Omega_\epsilon^f)}^2+\norm{\mathbf{w}}_{H^1(\Omega_\epsilon^s)}^2+\norm{\llbracket\mathbf{w}\rrbracket_s^f}^2_{L^2(\Gamma_\epsilon)}.
\end{equation}
Notice that $V$ is a closed subspace of $H^1(\Omega_\epsilon^f\cup \Omega_\epsilon^s)$ by the trace theorem.

The counterpart of $V$ for functions defined in the unit cell $Y$ is given by:
\begin{align*}
 V_Y&:=\left\{\mathbf{w}\in H^1(Y_f\cup Y_s)\, : \mathbf{w} \mbox{ is $Y$-periodic}\right\},\\
\norm{\mathbf{w}}_{V_Y}^2&:=\norm{\mathbf{w}}_{H^1(Y_f)}^2+\norm{\mathbf{w}}_{H^1(Y_s)}^2+\norm{\llbracket\mathbf{w}\rrbracket_s^f}^2_{L^2(\Gamma)}.
\end{align*}
The role of the interface term in the norm will be made clear later.
As will be revealed in Theorem \ref{uniform_bound}, the following Hilbert space is also needed in the analysis.
\begin{equation}
	\label{E0def}
	E_0(\Omega_\epsilon^s\cup\Omega_\epsilon^f)=\big\{\mathbf{w}\,:\,\mathbf{w}\in L^2(\Omega_\epsilon^s\cup\Omega_\epsilon^f),\,{\rm div}(\mathbf{w}) \in L^2(\Omega_\epsilon^s\cup\Omega_\epsilon^f),\,\mathbf{w}\cdot\mathbf{n}_\Omega=0 \text{ on } \partial\Omega\big\},
\end{equation} 
where $\mathbf{n}_\Omega$ represents the unit normal on $\partial\Omega$ pointing outward from $\Omega$, equipped with the inner product: $$ <{\mathbf{w}_1},{\mathbf{w}_2}>_{E_0(\Omega_\epsilon^s\cup\Omega_\epsilon^f)}=\int_{\Omega}\mathbf{w}_1\cdot \overline{\mathbf{w}_2} \,d\mathbf{x}+\int_{\Omega_\epsilon^s\cup\Omega_\epsilon^f}{\rm div}(\mathbf{w}_1) \overline{{\rm div}(\mathbf{w}_2)}\, d\mathbf{x},$$ for $\mathbf{w}_1,\,\mathbf{w}_2\in E_0(\Omega_\epsilon^s\cup\Omega_\epsilon^f)$.  The norm induced by this inner product is denoted by $\| \cdot  \|_{E_0(\Omega_\epsilon^s\cup\Omega_\epsilon^f)}$.{The Laplace transform of the fluid motion with respect to the solid  will be shown to be in the following space, }
\begin{equation}
\label{spaceW}
{W}:=\left\{\mathbf{w}\in V_Y \,:\,\mathbf{w}=0\text{ on }Y_s,\text{ and }{\rm div}_\mathbf{y}\mathbf{w}=0\right\}.
\end{equation}  
Note that $W$ is a closed vector subspace of $V_Y$.
\subsection{{Derivation of the variational problem}}
{The variational formulation in the function spaces mentioned above is derived in this section.} 
\vspace{0.1in}

For the solid phase, (\ref{S1}) and (\ref{S2}) lead to:
\begin{align*}
&\int_{\Omega_\epsilon^s}\rho^s\frac{\partial^2 u_\epsilon^i}{\partial t^2}\overline{w^i}d\mathbf{x}\\
&\quad=\int_{\Omega_\epsilon^s}\overline{w^i}f^id\mathbf{x}-\int_{\Omega_\epsilon^s}\frac{\partial \overline{w^i}}{\partial x_j}\left(a_{ijkl}E_{kl}(\mathbf{u}_\epsilon)\right)d\mathbf{x}
	+\int_{\Gamma_\epsilon}\overline{w^i}a_{ijkl}E_{kl}(\mathbf{u}_\epsilon)n_s^jd\sigma_\epsilon(\mathbf{x}),
\end{align*}
for all $\mathbf{w}\in V$, where $\mathbf{n}_s$ is the unit normal vector of $\Gamma_\epsilon$ pointing out of $\Omega_\epsilon^s$ (toward the fluid part).  By the symmetry of $a_{ijkl}$ in \eqref{symmetry}, we have 
\begin{equation*}
	\int_{\Omega_\epsilon^s}a_{ijkl}\frac{\partial u^k}{\partial x_l}\frac{\partial \overline{v^i}}{\partial x_j}d\mathbf{x}=\int_{\Omega_\epsilon^s}a_{ijkl}E_{kl}(\mathbf{u})\overline{E_{ij}(\mathbf{v})}d\mathbf{x}.
\end{equation*}
Therefore, the above can be rewritten as:
\begin{align}
\label{vpsolid}
	\int_{\Omega_\epsilon^s}\rho^s\frac{\partial^2 u_\epsilon^i}{\partial t^2}\overline{w^i}d\mathbf{x}&=\int_{\Omega_\epsilon^s}\overline{w^i}f^id\mathbf{x}-\int_{\Omega_\epsilon^s}a_{ijkl}E_{kl}(\overline{\mathbf{w}})E_{kl}(\mathbf{u}_\epsilon)d\mathbf{x}\notag\\
	&\quad-\int_{\Gamma_\epsilon}\overline{w^i}a_{ijkl}E_{kl}(\mathbf{u}_\epsilon)n^jd\sigma_\epsilon(\mathbf{x}),
\end{align}
where the minus sign in the last integral is due to the fact that $\mathbf{n}=-\mathbf{n}_s$.

For the fluid phase, (\ref{F1}) and (\ref{F8}) imply for all $\mathbf{w}\in V$, we have:
\begin{align}
\label{vpfluid}
	\int_{\Omega_\epsilon^f}\rho^f\frac{\partial^2 u^i_\epsilon}{\partial t^2}\overline{w^i}d\mathbf{x}
&=\int_{\Omega_\epsilon^f}\overline{w^i}f^id\mathbf{x}-\int_{\Omega_\epsilon^f}\left[c_0^2\rho^f(\nabla\cdot \mathbf{u}_\epsilon)(\overline{\nabla\cdot \mathbf{w}})\right.\notag\\
&\quad\left.+\epsilon^2\eta(\overline{\nabla\cdot \mathbf{w}})\left(\nabla\cdot\frac{\partial \mathbf{u}_\epsilon}{\partial t}\right)+2\epsilon^2\mu\frac{\partial\overline{w^i}}{\partial x_j}E_{ij}\left(\frac{\partial \mathbf{u}_\epsilon}{\partial t}\right)\right]d\mathbf{x}\notag\\
&\quad+\int_{\Gamma_\epsilon}\left[c_0^2\rho^fn^j\overline{w^j}(\nabla\cdot \mathbf{u}_\epsilon)+\epsilon^2\eta n^j\overline{w^j}\left(\nabla\cdot\frac{\partial \mathbf{u}_\epsilon}{\partial t}\right)\right.\notag\\
&\quad\left.+2\mu\epsilon^2n^j\overline{w^i}E_{ij}\left(\frac{\partial \mathbf{u}_\epsilon}{\partial t}\right)\right]d\sigma_\epsilon(\mathbf{x}).
\end{align}
Observe that $a_{ijkl}E_{kl}(\mathbf{u}_\epsilon)n^j=(\tend{\sigma}^s\mathbf{n})^i$ is the total solid stress acting on the interface and, similarly, the terms in the boundary integral in (\ref{vpfluid}) can be regarded as $(\tend{\sigma}^f\mathbf{n})^i$.  Summing the boundary integrals in (\ref{vpsolid}) and (\ref{vpfluid}), we can use (\ref{contnormstress}) to obtain:
\begin{align}
\label{slipcondit} 
-\int_{\Gamma_\epsilon^s}(\tend{\sigma}^s\mathbf{n})^i\overline{w^i}d\sigma_\epsilon(\mathbf{x})+\int_{\Gamma_\epsilon^f}(\tend{\sigma}^f\mathbf{n})^i\overline{w^i}d\sigma_\epsilon(\mathbf{x})
&=\int_{\Gamma_\epsilon}(\tend{\sigma}^f\mathbf{n})^i\left(\overline{\llbracket\mathbf{w}\rrbracket_s^f}\right)^id\sigma_\epsilon(\mathbf{x})\notag\\
&=\epsilon\int_{\Gamma_\epsilon}\alpha\left\llbracket\frac{\partial \mathbf{u_\epsilon}}{\partial t}\right\rrbracket_s^f\cdot\overline{\llbracket\mathbf{w}\rrbracket_s^f}\,d\sigma_\epsilon(\mathbf{x}).
\end{align}

From (\ref{vpsolid}), (\ref{vpfluid}) and (\ref{slipcondit}), the variational formulation for our problem is as follows: 

Find $\mathbf{u_\epsilon}$, function of $t$ with values in $V$, such that:
\begin{align*}
	\int_{\Omega}f^i\overline{w^i}d\mathbf{x}&=\int_{\Omega_\epsilon^s}\rho^s\frac{\partial^2 u_\epsilon^i}{\partial t^2}\overline{w^i}d\mathbf{x}+\int_{\Omega_\epsilon^f}\rho^f\frac{\partial^2 u_\epsilon^i}{\partial t^2}\overline{w^i}d\mathbf{x}\notag\\
&\quad+\int_{\Omega_\epsilon^s}a_{ijkl}E_{kl}(\mathbf{u_\epsilon})\overline{E_{ij}(\mathbf{w})}d\mathbf{x}+\int_{\Omega_\epsilon^f}\gamma(\nabla\cdot \mathbf{u_\epsilon})(\overline{\nabla\cdot \mathbf{w}})d\mathbf{x}\notag\\
&\quad+\epsilon^2\int_{\Omega_\epsilon^f}\left[\eta(\overline{\nabla\cdot \mathbf{w}})\left(\nabla\cdot\frac{\partial \mathbf{u_\epsilon}}{\partial t}\right)+2\mu\overline{E_{ij}(\mathbf{w})}E_{ij}\left(\frac{\partial \mathbf{u_\epsilon}}{\partial t}\right)\right]d\mathbf{x}\notag\\
&
\quad+\epsilon\int_{\Gamma_\epsilon}\alpha\left\llbracket\frac{\partial \mathbf{u_\epsilon}}{\partial t}\right\rrbracket_s^f\cdot\overline{\llbracket\mathbf{w}\rrbracket_s^f}d\sigma_\epsilon(\mathbf{x}),
\end{align*}
for all $\mathbf{w}\in V$; or equivalently:
\begin{align}
\label{TVF5}
	\int_{\Omega}f^i\overline{w^i}d\mathbf{x}&=\int_{\Omega_\epsilon^s}\rho^s\frac{\partial^2 u_\epsilon^i}{\partial t^2}\overline{w^i}d\mathbf{x}+\int_{\Omega_\epsilon^f}\rho^f\frac{\partial^2 u_\epsilon^i}{\partial t^2}\overline{w^i}d\mathbf{x}+c^{\epsilon}(\mathbf{u_\epsilon},\mathbf{w})\notag\\
&
\quad+\epsilon^2b^\epsilon\left(\frac{\partial \mathbf{u_\epsilon}}{\partial t},\mathbf{w}\right)+\epsilon\int_{\Gamma_\epsilon}\alpha\left\llbracket\frac{\partial \mathbf{u_\epsilon}}{\partial t}\right\rrbracket_s^f\cdot\overline{\llbracket\mathbf{w}\rrbracket_s^f}\,d\sigma_\epsilon(\mathbf{x}),
\end{align}
for all $\mathbf{w}\in V$, where $\mathbf{u_\epsilon}(0)=\displaystyle\frac{\partial \mathbf{u_\epsilon}}{\partial t}(0)=0$ and the sesquilinear forms $b^\epsilon$  and $c^\epsilon$ are defined as
\begin{align}
	\label{c-eps}
	b^\epsilon(\mathbf{u},\mathbf{v})&=\int_{\Omega_\epsilon^f}\left[\eta\left(\nabla\cdot\mathbf{u}\right)(\overline{\nabla\cdot \mathbf{v}})+2\mu E_{ij}\left(\mathbf{u}\right)\overline{E_{ij}(\mathbf{v})}\right]d\mathbf{x},\notag\\
c^\epsilon(\mathbf{u},\mathbf{v})&=\int_{\Omega_\epsilon^s}a_{ijkl}E_{kl}(\mathbf{u})\overline{E_{ij}(\mathbf{v})}d\mathbf{x}+\int_{\Omega_\epsilon^f}\gamma(\nabla\cdot \mathbf{u})(\overline{\nabla\cdot \mathbf{v}})d\mathbf{x}.
\end{align}

\section{Well-Posedness of the Variational Formulation}
\label{EUProof}

{In this section, the analysis of the variational problem \eqref{TVF5} is carried out in the Laplace transformed domain.} 

Let $\mathbf{\hat{v}}(\lambda)$ be the Laplace transform of a function $\mathbf{v}(t)$. 
The variational formulation of problem (\ref{TVF5}) in the Laplace transform domain {for a fixed $\lambda$} reads as follows (for the ease of notation, we omit the argument $\lambda$ in $\mathbf{\hat{u}_\epsilon}(\lambda)$ and $\mathbf{\hat{f}}(\lambda)$) 
\vspace{0.1in}

	Find $\mathbf{\hat{u}_\epsilon}\in V$ such that, for all $\mathbf{w}\in V$ {the following equation is satisfied:}
\begin{equation}
\label{LVF}
\int_{\Omega}\mathbf{\hat{f}}\cdot\overline{\mathbf{w}}d\mathbf{x}=a^\epsilon(\mathbf{\hat{u}_\epsilon},\mathbf{w})
\end{equation}
where the form $a^\epsilon(\mathbf{\hat{u}_\epsilon},\mathbf{w})$ is defined as follows:
\begin{align*}
a^\epsilon(\mathbf{\hat{u}_\epsilon},\mathbf{w})&:=
\lambda^2\int_{\Omega_\epsilon^s}\rho^s \mathbf{\hat{u}_\epsilon}\cdot\overline{\mathbf{w}}d\mathbf{x}+\lambda^2\int_{\Omega_\epsilon^f}\rho^f\mathbf{\hat{u}_\epsilon}\cdot\overline{\mathbf{w}}d\mathbf{x}+c^{\epsilon}(\mathbf{\hat{u}_\epsilon},\mathbf{w})+\lambda\epsilon^2b^\epsilon\left(\mathbf{\hat{u}_\epsilon},\mathbf{w}\right)\\
&
\quad+\lambda\,\epsilon\int_{\Gamma_\epsilon}\alpha\llbracket\mathbf{\hat{u}_\epsilon}\rrbracket_s^f\cdot \overline{\llbracket\mathbf{w}\rrbracket}_s^f\,d\sigma_\epsilon(\mathbf{x}).
\end{align*}

The main result in this section is the following theorem.
{
\begin{theorem}
\label{existence_lap}
For any fixed $\epsilon>0$ and $\lambda$ such that $Re\lambda>\lambda_0>0$, with $\lambda_0$ large enough, the variational problem \eqref{LVF} has a unique solution.
\end{theorem}
}

The following lemma {plays a crucial role in the proof of Theorem\ref{existence_lap}.}
\begin{lemma}
\label{rem:bsp}
Suppose $\partial \Omega^s_\epsilon \cap \partial \Omega\ne\emptyset $.  Then, for all $\mathbf{w}\in V$, there exists $K=K(\epsilon)>0$ such that the following estimate holds:
\begin{align}
	\label{1.15}
&\int_{\Omega^s_\epsilon}w^i\overline{w^i}d\mathbf{x}+\int_{\Omega^f_\epsilon}w^i\overline{w^i}d\mathbf{x} \\
&\leq K\left(\int_{\Omega^s_\epsilon}  E_{ij}(\mathbf{w})\overline{E_{ij}(\mathbf{w})}d\mathbf{x}+ \int_{\Omega^f_\epsilon}  E_{ij}(\mathbf{w})\overline{E_{ij}(\mathbf{w})}d\mathbf{x}+\int_{\Gamma_\epsilon} \left|\llbracket\mathbf{w}\rrbracket^f_s \right|^2 d\sigma_\epsilon(\mathbf{x})  \right).\notag
\end{align}
\end{lemma}
\begin{proof}
We prove this by contradiction. Suppose (\ref{1.15}) is not true.  Then, there exists a sequence $\left\{\mathbf{w}^k\right\}_{k=1}^{\infty}$ in V, with $\norm{\mathbf{w}^k}_{\mathbf{L^2}( \Omega_\epsilon^f \times \Omega_\epsilon^s )}=1$, satisfying: $$\int_{  \Omega_\epsilon^f \times \Omega_\epsilon^s}E_{ij}(\mathbf{w}^k)\overline{E_{ij}(\mathbf{w}^k)}dx+\int_{\Gamma_\epsilon} \left|\llbracket\mathbf{w}\rrbracket^f_s \right|^2 d\sigma_\epsilon(x) \rightarrow0,\hspace{4mm}\text{ as $k\rightarrow\infty$.}$$
Hence, by {the Korn's inequality in} Lemma \ref{lem:Korn}, there exists $\mathbf{w^*}\in V$ such that $\mathbf{w}^k\rightarrow\mathbf{w^*}$ weakly in $V$ and $\mathbf{w}^k\rightarrow\mathbf{w^*}$ strongly in $L^2(\Omega)$.  Therefore, $\norm{\mathbf{w^*}}_{L^2(\Omega)}=1$ and $\norm{\mathbf{w}^k}_{V}\leq C$.  By Proposition~1.1 on page~8 of \cite{SanchezPalencia1980}, we have $\norm{\mathbf{w^*}}_{ V}\leq\liminf_k\norm{\mathbf{w}^k}_{V}$.  From the latter, it follows that:
\begin{align*}
&1+\int_{\Omega_\epsilon^f \times \Omega_\epsilon^s}E_{ij}(\mathbf{w^*})\overline{E_{ij}(\mathbf{w^*})}dx+ \int_{\Gamma_\epsilon} \left|\llbracket\mathbf{w^*}\rrbracket^f_s \right|^2 d\sigma_\epsilon(x) \nonumber\\
&\quad\leq\liminf_k\left(1+\int_{  \Omega_\epsilon^f \times \Omega_\epsilon^s}E_{ij}(\mathbf{w}^k)\overline{E_{ij}(\mathbf{w}^k)}dx + \int_{\Gamma_\epsilon} \left|\llbracket\mathbf{w}^k\rrbracket^f_s \right|^2 d\sigma_\epsilon(x)\right),
\end{align*}
which, in turn, implies that:
\begin{equation}
\label{zero-jump}
    \int_{ \Omega_\epsilon^f \times \Omega_\epsilon^s }E_{ij}(\mathbf{w^*})\overline{E_{ij}(\mathbf{w^*})}dx=0\,\,\text{  and   }\,\, \int_{\Gamma_\epsilon} \left|\llbracket\mathbf{w^*}\rrbracket^f_s \right|^2 d\sigma_\epsilon(x)=0.
\end{equation}
Therefore $\mathbf{w^*}$ is a rigid body motion in each phase.  Since  $\left.\mathbf{w^*}\right|_{ \partial \Omega_\epsilon^s  \cap\partial \Omega}=\mathbf{0}$, we have that 
$\mathbf{w^*}=\mathbf{0}$ in $ \Omega_\epsilon^s$ . By the interface integral in \eqref{zero-jump}, we must also have  $\mathbf{w^*}=\mathbf{0}$ in $ \Omega_\epsilon^f$. This contradicts $\norm{\mathbf{w^*}}_{L^2(\Omega)}=1$.
\end{proof}
The proof above demonstrates the importance of including the interface jump in the norm of the space $V$ in (\ref{V}).  Without the interface term, the lemma would not be true. 

{With this lemma, we are ready to prove Theorem \ref{existence_lap}}. 
\begin{proof}(of Theorem \ref{existence_lap})
{The fact that $\lambda$ is a complex number and the appearance of various orders of $\lambda$ in the expression of $a^\epsilon(\cdot,\cdot)$ prevent a direct application of the Lax-Milgram lemma (Lemma \ref{laxmil}). Noticing that $Re(\lambda)$ and $Re(\frac{1}{\lambda})$ have the same sign, we recast the variational problem (\ref{LVF}) to an equivalent problem by dividing} both sides of (\ref{LVF}) with $\lambda\ne 0$:  
\begin{align}
\mbox{Find } \mathbf{\hat{u}_\epsilon}\in V &\mbox{ such that, for all } \mathbf{w}\in V,  \mbox{ we have:}\notag\\
\label{LVF2}
\frac{1}{\lambda}a^\epsilon(\mathbf{\hat{u}_\epsilon},\mathbf{w})&=\lambda\int_{\Omega_\epsilon^s}\rho^s \hat{u}_\epsilon^i\overline{w^i}d\mathbf{x}+\lambda\int_{\Omega_\epsilon^f}\rho^f\hat{u}_\epsilon^i\overline{w^i}d\mathbf{x}+\frac{1}{\lambda}c^{\epsilon}(\mathbf{\hat{u}_\epsilon},\mathbf{w})\notag\\
&\quad+\epsilon^2b^{\epsilon}(\mathbf{\hat{u}_\epsilon},\mathbf{w})+\epsilon\int_{\Gamma_\epsilon}\alpha\llbracket\mathbf{\hat{u}_\epsilon}\rrbracket_s^f\cdot \overline{\llbracket\mathbf{w}\rrbracket}_s^f\,d\sigma_\epsilon(\mathbf{x})\notag\\
&=\frac{1}{\lambda}\int_{\Omega}\hat{f}^i\overline{w^i}d\mathbf{x}.
\end{align}

{To show the coercivity of} $\left(\frac{1}{\lambda}\right)a^\epsilon(\mathbf{w},\mathbf{w})$, for $\mathbf{w}\in V$, with $Re(\lambda)>0$ and $\epsilon>0$, we observe that the properties of the coefficients $a_{ijkl}$, $\eta$ and $\mu$ in (\ref{a-coercive}) and (\ref{etamu}) imply the following inequality
\begin{align}
\label{coercivity}
&{\rm Re}\left( \frac{1}{\lambda}c^\epsilon(\mathbf{w},\mathbf{w})+\epsilon^2b^\epsilon\left(\mathbf{w},\mathbf{w}\right) +\epsilon\int_{\Gamma_\epsilon}\alpha\left|\llbracket\mathbf{w}\rrbracket_s^f\right|^2\,d\sigma_\epsilon(\mathbf{x})\right)\notag\\
& >Re\left(\frac{1}{\lambda}\right)\int_{\Omega_\epsilon^s}cE_{ij}(\mathbf{w})\overline{E_{ij}(\mathbf{w})}d\mathbf{x}\notag\\
&\quad+\epsilon^2\mu\left\{\int_{\Omega_\epsilon^f}\left[-\frac{ 2}{3}|\nabla\cdot\mathbf{w}|^2+ E_{ij}(\mathbf{w})\overline{E_{ij}(\mathbf{w})}\right]d\mathbf{x}.\right.\notag\\
&\quad\left.+\int_{\Omega_\epsilon^f} E_{ij}(\mathbf{w})\overline{E_{ij}(\mathbf{w})}d\mathbf{x}\right\}  +\epsilon\alpha \int_{\Gamma_\epsilon}\left|\llbracket\mathbf{w}\rrbracket_s^f\right|^2\,d\sigma_\epsilon(\mathbf{x})\notag\\
&\geq\min( Re\left(\frac{c}{\lambda}\right),\epsilon^2\mu , \epsilon \alpha) \left(\int_{\Omega_\epsilon^s\times\Omega_\epsilon^f}E_{ij}(\mathbf{w})\overline{E_{ij}(\mathbf{w})}d\mathbf{x} + \int_{\Gamma_\epsilon}\left|\llbracket\mathbf{w}\rrbracket_s^f\right|^2\,d\sigma_\epsilon(\mathbf{x})\right).
\end{align}
%
We note that Lemma~\ref{lem:Korn} implies that there exist $\gamma_s',\,\gamma_f'>0$ such that:
\begin{eqnarray}
\label{lemma4-1-s}
\int_{\Omega_\epsilon^s}E_{ij}(\mathbf{w})\overline{E_{ij}(\mathbf{w})}d\mathbf{x}+\int_{\Omega_\epsilon^s}\mathbf{w}\cdot\overline{\mathbf{w}}d\mathbf{x}\geq\gamma_s'\norm{\mathbf{w}}^2_{ H^1(\Omega_\epsilon^s)}\\
\label{lemma4-1-f}
\int_{\Omega_\epsilon^f}E_{ij}(\mathbf{w})\overline{E_{ij}(\mathbf{w})}d\mathbf{x}+\int_{\Omega_\epsilon^f}\mathbf{w}\cdot\overline{\mathbf{w}}d\mathbf{x}\geq\gamma_f'\norm{\mathbf{w}}^2_{ H^1(\Omega_\epsilon^f)}.
\end{eqnarray}
Also, as long as $\partial \Omega_\epsilon^s\cap\partial\Omega\neq\emptyset$ , by Lemma~\ref{rem:bsp}, there exists $K>0$ such that:
\begin{equation}
\label{coerSolid}
\int_{\Omega_\epsilon^s  \times \Omega_\epsilon^f }w^i\overline{w^i}d\mathbf{x}\leq K \left(\int_{\Omega_\epsilon^s  \times \Omega_\epsilon^f}E_{ij}(\mathbf{w})\overline{E_{ij}(\mathbf{w})} d\mathbf{x}  +\int_{\Gamma_\epsilon} \left|\llbracket\mathbf{w}\rrbracket^f_s \right|^2 d\sigma_\epsilon(\mathbf{x}) \right),
\end{equation}
for all $ \mathbf{w}\in V$.  

Using \eqref{lemma4-1-s}, \eqref{lemma4-1-f}, and (\ref{coerSolid}) in (\ref{coercivity}), we have:
\begin{align}
\label{coercivity2}
&{\rm Re}\left(\left(\frac{1}{\lambda}\right)c^\epsilon(\mathbf{w},\mathbf{w})+\epsilon^2 b^\epsilon\left(\mathbf{w},\mathbf{w}\right)+\epsilon\int_{\Gamma_\epsilon}\alpha\left|\llbracket\mathbf{w}\rrbracket_s^f\right|^2\,d\sigma_\epsilon(\mathbf{x})\right)\notag\\
&\geq\min\left(Re\left(\frac{c}{\lambda}\right),\epsilon^2\mu, \epsilon\alpha\right)\left[\frac{1}{2}\int_{\Omega_\epsilon^s \times \Omega_\epsilon^f}E_{ij}(\mathbf{w})\overline{E_{ij}(\mathbf{w})}d\mathbf{x}\right.\notag\\
&\quad\left.  +\frac{1}{2 K}\int_{\Omega_\epsilon^s \times \Omega_\epsilon^f}w^i\overline{w^i}d\mathbf{x}+\frac{1}{2}\int_{\Gamma_\epsilon}\left|\llbracket\mathbf{w}\rrbracket_s^f\right|^2\,d\sigma_\epsilon(\mathbf{x})\right]\notag\\
&\geq  C' \| \mathbf{w}\|^2_V, 
\end{align}
where $C':=\min(Re\left(\frac{c}{\lambda}\right),\epsilon^2\mu ,\epsilon\alpha)\cdot\min\left(\frac{1}{2},\frac{1}{ 2K}\right)\cdot\min\left(\gamma_f',\gamma_s',\frac{1}{2}\right)$.
Therefore, for all $\mathbf{w}\in V$, we have:
\begin{align*}
&{\rm Re}\left(\frac{1}{\lambda}a^\epsilon(\mathbf{w}, \mathbf{w})\right)
	\geq C'\norm{\mathbf{w}}_V^2.
\end{align*}
This proves the coercivity of the sesquilinear form in  \eqref{LVF2}. The boundedness of this form can be checked easily by a repeated application of the Cauchy-Schwarz inequality. Therefore, by Lemma~\ref{laxmil} (Lax-Milgram lemma), there exists a unique solution of \eqref{LVF2} and, hence, of \eqref{LVF}, for any fixed $\epsilon>0$ and $\lambda$, with $Re(\lambda)>0$.
\end{proof}
\section{Uniform bounds and the weak limit}
\label{sec:premres}
{We have shown that for any fixed $\lambda$ with $\lambda>0$, there is a unique solution $\mathbf{\hat{u}}_\epsilon$ for each $\epsilon>0$. In order to apply the compactness results of the two-scale convergence, cf. Definition \ref{2cale-conv}, we need to estimate the  sequence $\{\mathbf{\hat{u}}_\epsilon\}$ and their derivatives so as to derive the bounds which are uniform in $\epsilon$.}

{
The main result in this section is the following theorem regarding the uniform bounds of the  sequence of solutions $\{\mathbf{\hat{u}}_\epsilon\}$.
\begin{theorem}
\label{uniform_bound}
 For every fixed $\lambda$ such that $Re(\lambda)>0$, the  sequence of solutions  $\{\mathbf{\hat{u}}_\epsilon\}$ satisfies the following estimates.
\begin{eqnarray}
&&\norm{\mathbf{\hat{u}_\epsilon}}_{\mathbf{L^2}(\Omega)}\leq C \hspace{3mm}\forall\epsilon,\label{unifbounduep}\\
\label{unifboundu}
&&\epsilon\norm{\mathbf{\hat{u}_\epsilon}}_{V}\leq C\,,\hspace{3mm}\forall\epsilon,\\
\label{unifbounddivu}
&&\norm{{\rm div}\,\mathbf{\hat{u}_\epsilon}}_{L^2(\Omega_\epsilon^s\cup \Omega_\epsilon^f)}\leq C\,,\hspace{3mm}\forall\epsilon,\\
\label{uni_deri_solid}
&&\norm{\nabla\mathbf{u}_\epsilon}_{\mathbf{L}^2(\Omega_\epsilon^s)}\leq C,\hspace{3mm}\text{ for all $0<\epsilon<\epsilon_o$.}	
\end{eqnarray}
\end{theorem}
}
{
From this theorem, we see that the restriction of $\{\mathbf{\hat{u}}_\epsilon\}$ to the solid phase $\Omega_\epsilon^s$ are uniformly bounded in $\mathbf{H}^1(\Omega_\epsilon^s)$ while the restriction to the fluid phase are only bounded uniformly in $\mathbf{H}_{\rm{div}}(\Omega_\epsilon^f)$. This prompts the introduction of the well known space $E_0(\Omega_\epsilon^s\cup\Omega_\epsilon^f)$ defined in \eqref{E0def}.
}
\vspace{0.1in}

{The following lemmas are essential in proving Theorem \ref{uniform_bound}.
}
\begin{lemma}
\label{time_bound}
Let $\rho^*=\min\left\{\rho_f,\rho_s\right\}$ and $\mathbf{z_\epsilon}(t)=e^{-rt}\mathbf{u_\epsilon}(t)$, where $r>0$ is a fixed real number and $\mathbf{u_\epsilon}$ is the solution of  (\ref{TVF5}).  We have, for $r>\max\left\{{0},\frac{m}{2}\right\}$, where $m$ is the growth rate of $\mathbf{f}$ defined in \eqref{f}, that:
\begin{equation}
\label{unifbound}
	\rho^*\norm{\frac{\partial \mathbf{z_\epsilon}}{\partial t}}^2_{L^2(\Omega)}+\rho^*\norm{\mathbf{z_\epsilon}}^2_{L^2(\Omega)}+c^{\epsilon}(\mathbf{z_\epsilon},\mathbf{z_\epsilon})+\epsilon^2b^\epsilon\left(\mathbf{z_\epsilon},\mathbf{z_\epsilon}\right)+\alpha\epsilon\norm{\llbracket\mathbf{z_\epsilon}\rrbracket_s^f}^2_{L^2(\Gamma_\epsilon)}\leq C,
\end{equation}
for almost all $0<t<\infty$ and for all $\epsilon>0$. Moreover,
\begin{equation}
\label{unifbound2}
	\norm{\mathbf{z_\epsilon}}^2_{L^2(\Omega)}+\norm{{\rm div}(\mathbf{z_\epsilon})}^2_{L^2(\Omega)}\leq C,
\end{equation}
for all $\epsilon$ and almost all $t>0$.
\end{lemma}
\begin{proof}
From the definition of $\mathbf{z_\epsilon}(t)$, we have:
\begin{eqnarray*}
\frac{\partial \mathbf{u_\epsilon}}{\partial t}&=&e^{rt}\left(\frac{\partial \mathbf{z_\epsilon}}{\partial t}+r\mathbf{z_\epsilon}\right),\\
\frac{\partial^2 \mathbf{u_\epsilon}}{\partial t^2}&=&e^{rt}\left(\frac{\partial^2 \mathbf{z_\epsilon}}{\partial t^2}+2r\frac{\partial \mathbf{z_\epsilon}}{\partial t}+r^2\mathbf{z_\epsilon}\right).
\end{eqnarray*}

Plugging these into (\ref{TVF5}), and taking $\mathbf{w}=\displaystyle\frac{\partial \mathbf{z_\epsilon}}{\partial t}$, we obtain: 
\begin{align*}
	e^{-rt}\int_{\Omega}\mathbf{f}\cdot\overline{\frac{\partial \mathbf{z_\epsilon}}{\partial t}}d\mathbf{x}
	&=\int_{\Omega}\rho\frac{\partial^2 \mathbf{z_\epsilon}}{\partial t^2}\cdot\overline{\frac{\partial \mathbf{z_\epsilon}}{\partial t}}d\mathbf{x}+2r\int_{\Omega}\rho\left|\frac{\partial \mathbf{z_\epsilon}}{\partial t}\right|^2d\mathbf{x}+r^2\int_{\Omega}\rho\mathbf{z_\epsilon}\cdot\overline{\frac{\partial \mathbf{z_\epsilon}}{\partial t}}d\mathbf{x}\notag\\
	&\quad+c^{\epsilon}\left(\mathbf{z_\epsilon},\frac{\partial \mathbf{z_\epsilon}}{\partial t}\right)+\epsilon^2b^\epsilon\left(\frac{\partial \mathbf{z_\epsilon}}{\partial t},\frac{\partial \mathbf{z_\epsilon}}{\partial t}\right)+r\epsilon^2b^\epsilon\left(\mathbf{z_\epsilon},\frac{\partial \mathbf{z_\epsilon}}{\partial t}\right)\notag\\
	&\quad+\epsilon\alpha\int_{\Gamma_\epsilon}\left|\left\llbracket\frac{\partial \mathbf{z_\epsilon}}{\partial t}\right\rrbracket_s^f\right|^2\,d\sigma_\epsilon(\mathbf{x})+r\epsilon\alpha\int_{\Gamma_\epsilon}\llbracket\mathbf{z_\epsilon}\rrbracket_s^f\cdot \overline{\left\llbracket\frac{\partial \mathbf{z_\epsilon}}{\partial t}\right\rrbracket}_s^f\,d\sigma_\epsilon(\mathbf{x}).
\end{align*}
%
%
Rearranging terms and applying the estimate of $\mathbf{f}$ in (\ref{f}) lead to:
\begin{align}
\label{TVF5-z2}
	&\frac{1}{2}\frac{d}{dt}\left[\int_{\Omega}\rho\left|\frac{\partial \mathbf{z_\epsilon}}{\partial t}\right|^2d\mathbf{x}+r^2\int_{\Omega}\rho|\mathbf{z_\epsilon}|^2d\mathbf{x}+c^{\epsilon}(\mathbf{z_\epsilon},\mathbf{z_\epsilon})\right.\notag\\
\
	&\left.\quad+r\epsilon^2b^\epsilon\left(\mathbf{z_\epsilon},\mathbf{z_\epsilon}\right)+r\epsilon\alpha\int_{\Gamma_\epsilon}\llbracket\mathbf{z_\epsilon}\rrbracket_s^f\cdot \overline{\llbracket\mathbf{z_\epsilon}\rrbracket}_s^f\,d\sigma_\epsilon(\mathbf{x})\right]\notag\\
	&=e^{-rt}\int_{\Omega}\mathbf{f}\cdot\overline{\frac{\partial \mathbf{z_\epsilon}}{\partial t}}d\mathbf{x}-2r\int_{\Omega}\rho\left|\frac{\partial \mathbf{z_\epsilon}}{\partial t}\right|^2d\mathbf{x}\notag\\
	&\quad-\epsilon^2b^\epsilon\left(\frac{\partial \mathbf{z_\epsilon}}{\partial t},\frac{\partial \mathbf{z_\epsilon}}{\partial t}\right)-\epsilon\alpha\int_{\Gamma_\epsilon}\left|\left\llbracket\frac{\partial \mathbf{z_\epsilon}}{\partial t}\right\rrbracket_s^f\right|^2\,d\sigma_\epsilon(\mathbf{x})\notag\\
	&\leq e^{-rt}\int_{\Omega}\mathbf{f}\cdot\overline{\frac{\partial \mathbf{z_\epsilon}}{\partial t}}d\mathbf{x}\notag\\
 &\leq  e^{-rt}\norm{\mathbf{f}}_{L^2(\Omega)}\norm{\frac{\partial \mathbf{z_\epsilon}}{\partial t}}_{L^2(\Omega)}\notag\\
	&\leq e^{-rt}\frac{ K^{1/2} e^{mt/2}}{\sqrt{\rho^*}}\left(\int_{\Omega}\rho\left|\frac{\partial \mathbf{z_\epsilon}}{\partial t}\right|^2dx\right)^{1/2}\notag\\
	&\leq e^{-rt}\frac{ K^{1/2} e^{mt/2}}{\sqrt{\rho^*}}\left(\int_{\Omega}\rho\left|\frac{\partial \mathbf{z_\epsilon}}{\partial t}\right|^2d\mathbf{x}+r^2\int_{\Omega_\epsilon}\rho|\mathbf{z_\epsilon}|^2d\mathbf{x}+c^{\epsilon}(\mathbf{z_\epsilon},\mathbf{z_\epsilon})\right.\notag\\
	&\left.\quad+r\epsilon^2b^\epsilon\left(\mathbf{z_\epsilon},\mathbf{z_\epsilon}\right)+r\epsilon\alpha\int_{\Gamma}\llbracket\mathbf{z_\epsilon}\rrbracket_s^f\cdot \overline{\llbracket\mathbf{z_\epsilon}\rrbracket}_s^f\,d\sigma_\epsilon(\mathbf{x})\right)^{1/2},
\end{align}
since $r$ and $\alpha$ are non-negative.  Due to the fact that: $$\frac{1}{2}\frac{d}{dt}(\cdot)=(\cdot)^{1/2}\frac{d}{dt}(\cdot)^{1/2},\hspace{4mm}\text{if $(\cdot)\geq0$,}$$  we can simplify (\ref{TVF5-z2}) to obtain
\begin{align*}
	&\frac{d}{dt}\left[\int_{\Omega}\rho\left|\frac{\partial \mathbf{z_\epsilon}}{\partial t}\right|^2d\mathbf{x}+r^2\int_{\Omega}\rho|\mathbf{z_\epsilon}|^2d\mathbf{x}+c^{\epsilon}(\mathbf{z_\epsilon},\mathbf{z_\epsilon})\right.\notag\\
	&\left.\quad+r\epsilon^2b^\epsilon\left(\mathbf{z_\epsilon},\mathbf{z_\epsilon}\right)+r\,\epsilon\,\alpha\int_{\Gamma_\epsilon}\llbracket\mathbf{z_\epsilon}\rrbracket_s^f\cdot \overline{\llbracket\mathbf{z_\epsilon}\rrbracket}_s^f\,d\sigma_\epsilon(\mathbf{x})\right]^{1/2}\notag\\
&\leq \frac{ K^{1/2} e^{(m/2-r)t}}{\sqrt{\rho^*}}.
\end{align*}
Because $\mathbf{z}_\epsilon|_{t=0}=\frac{\partial\mathbf{z}_\epsilon }{\partial t}\big|_{t=0}=0$, we have: 
\begin{align*}
	&\left(\int_{\Omega}\rho\left|\frac{\partial \mathbf{z_\epsilon}}{\partial t}\right|^2d\mathbf{x}+r^2\int_{\Omega}\rho|\mathbf{z_\epsilon}|^2d\mathbf{x}+c^{\epsilon}(\mathbf{z_\epsilon},\mathbf{z_\epsilon})+r\epsilon^2b^\epsilon\left(\mathbf{z_\epsilon},\mathbf{z_\epsilon}\right)\right.\notag\\
	&\left.\quad+r\epsilon\alpha\int_{\Gamma_\epsilon}\left[\mathbf{z_\epsilon}\right]_s^f\cdot \overline{\left[\mathbf{z_\epsilon}\right]}_s^f\,d\sigma_\epsilon(\mathbf{x})\right)^{1/2} \Big|_{t=T}
	\leq \frac{K^{1/2}}{\sqrt{\rho^*}\left(r-\frac{m}{2}\right)}\left(1-e^{(m/2-r)T}\right).
%
\end{align*}
Note that the bound does not depend on $\epsilon$. Therefore, for $\mathbf{z}_\epsilon$ with $r>\max(0,\frac{m}{2})$, \eqref{unifbound} must be true for $t>0$ a.e. and for all $\epsilon>0$.
Note that, for all $T>0$ and $r>\frac{m}{2}$, we have $0<e^{\frac{m-2r}{2}T}<e^0=1$.  This means: 
\begin{align*}
	&\int_{\Omega}\rho\left|\frac{\partial \mathbf{z_\epsilon}}{\partial t}\right|^2d\mathbf{x}+r^2\int_{\Omega}\rho|\mathbf{z_\epsilon}|^2d\mathbf{x}+c^{\epsilon}(\mathbf{z_\epsilon},\mathbf{z_\epsilon})\\
\
	&\quad+r\epsilon^2b^\epsilon\left(\mathbf{z_\epsilon},\mathbf{z_\epsilon}\right)+r\epsilon\alpha\int_{\Gamma_\epsilon}\llbracket\mathbf{z_\epsilon}\rrbracket_s^f\cdot \overline{\llbracket\mathbf{z_\epsilon}\rrbracket}_s^f\,d\sigma_\epsilon(\mathbf{x})\notag
\end{align*}
is uniformly bounded with respect to time $t$. 
{The bound stated in \eqref{unifbound2} then follows as a consequence of Lemma \ref{time_bound}, the definition of $c^\epsilon$ (see (\ref{c-eps})) and  (\ref{a-coercive}).}
\end{proof}
%
\begin{lemma}
For $r>\max({0},\frac{m}{2})$, we can extract a subsequence such that
\begin{equation*}
	\mathbf{z_\epsilon}\rightarrow \mathbf{z_0} \hspace{3mm} \text{ in $L^{\infty}(0,+\infty;E_0(\Omega_\epsilon^s\cup\Omega_\epsilon^f))$-weak star.}
\end{equation*}
	{Moreover, letting $\mathbf{u_0}(t):=\mathbf{z_0}(t)e^{rt}$, then there exists a subsequence of $\{\mathbf{u}_\epsilon\}$, denoted by the same symbol, which converges as follows}
\begin{align}
	\mathbf{\hat{u}_\epsilon}(\lambda)&\rightarrow \mathbf{\hat{u}_0}(\lambda)  &\text{ in \quad $E_0(\Omega_\epsilon^s\cup\Omega_\epsilon^f)$-weak for any $\lambda\in\mathbb{C}$, ${\rm Re}(\lambda)\ge\lambda_0>r$.}	\label{weakLap}\\
	\mathbf{u_\epsilon}&\rightarrow \mathbf{u_0}  &\text{ in \quad  $\mathbf{L^{\infty}}(0,T;E_0(\Omega_\epsilon^s\cup\Omega_\epsilon^f))$-weak star } \mbox{for any $T>0$.}		\label{weakstaru}
\end{align}
\end{lemma}	
\begin{proof}
	By virtue of (\ref{unifbound2}), the sequence $\left\{\mathbf{z_\epsilon}\right\}$ remains uniformly bounded in the space $L^{\infty}(0,+\infty;E_0(\Omega_\epsilon^s\cup\Omega_\epsilon^f))$, for all $r>\max\left\{{0},\frac{m}{2}\right\}$, i.e. for all $\phi\in L^1(0,+\infty,E_0(\Omega_\epsilon^s\cup\Omega_\epsilon^f))$, we have:
$$\lim_{\epsilon\rightarrow0}\int_0^\infty\langle \mathbf{z_\epsilon}(t,\cdot),\phi(t,\cdot)\rangle_{E_0(\Omega_\epsilon^s\cup\Omega_\epsilon^f)}dt=\int_0^\infty\langle \mathbf{z_0}(t,\cdot),\phi(t,\cdot)\rangle_{E_0(\Omega_\epsilon^s\cup\Omega_\epsilon^f)}dt;$$ and if $\phi(\mathbf{x},t)=e^{-st}\psi(\mathbf{x})$,  $s>0$, this is equivalent to: \begin{align*}\lim_{\epsilon\rightarrow0}\langle \mathbf{\hat{z}_\epsilon(s)},\psi(\mathbf{x})\rangle_{E_0(\Omega_\epsilon^s\cup\Omega_\epsilon^f)}&=\lim_{\epsilon\rightarrow0}\langle \int_0^\infty e^{-st}\mathbf{z_\epsilon}(t,\cdot)dt,\psi(\mathbf{x})\rangle_{E_0(\Omega_\epsilon^s\cup\Omega_\epsilon^f)}\\
    &=\langle \mathbf{\hat{z}_0}(s),\psi(\mathbf{x})\rangle_{E_0(\Omega_\epsilon^s\cup\Omega_\epsilon^f)}.
\end{align*}

By letting $\mathbf{u_0}(t)=\mathbf{z_0}(t)e^{rt}$, with $r>\max\{{0},m/2\}$, \eqref{weakLap} and \eqref{weakstaru} can be deduced.
\end{proof}
\begin{lemma}\label{1-coercive}
There exists a positive constant $C$, independent of $\epsilon$, such that:
\begin{equation*}
b^\epsilon\left(\mathbf{\hat{u}_\epsilon},\mathbf{\hat{u}_\epsilon}\right)+c^{\epsilon}(\mathbf{\hat{u}_\epsilon},\mathbf{\hat{u}_\epsilon})+\norm{\llbracket{\mathbf{\hat{u}}_\epsilon}\rrbracket_s^f}^{2}_{L^2(\Gamma_\epsilon)}\ge C \| \mathbf{\hat{u}_\epsilon}\|_V^2.
\end{equation*}
\end{lemma}
\begin{proof}
For all $\mathbf{w}\in V$, we have:
\begin{align}
&c^\epsilon(\mathbf{w},\mathbf{w})+b^\epsilon\left(\mathbf{w},\mathbf{w}\right) {+\int_{\Gamma_\epsilon}\left|\llbracket\mathbf{w}\rrbracket_s^f\right|^2\,d\sigma_\epsilon(\mathbf{x})} \label{coersive-one}\\
&\quad\geq c\int_{\Omega_\epsilon^s}E_{ij}(\mathbf{w})\overline{E_{ij}(\mathbf{w})}d\mathbf{x}\notag\\
&\qquad+\mu\left\{\int_{\Omega_\epsilon^f}\left[-\frac{ 2}{3}|\nabla\cdot\mathbf{w}|^2+ E_{ij}(\mathbf{w})\overline{E_{ij}(\mathbf{w})}\right]d\mathbf{x}.\right.\notag\\
&\qquad\left.+\int_{\Omega_\epsilon^f} E_{ij}(\mathbf{w})\overline{E_{ij}(\mathbf{w})}d\mathbf{x}\right\}+ \int_{\Gamma_\epsilon}\left|\llbracket\mathbf{w}\rrbracket_s^f\right|^2\,d\sigma_\epsilon(\mathbf{x})\notag\\
&\quad\geq c\int_{\Omega_\epsilon^s}E_{ij}(\mathbf{w})\overline{E_{ij}(\mathbf{w})}d\mathbf{x}
+\mu \int_{\Omega_\epsilon^f} E_{ij}(\mathbf{w})\overline{E_{ij}(\mathbf{w})}d\mathbf{x}+ \int_{\Gamma_\epsilon}\left|\llbracket\mathbf{w}\rrbracket_s^f\right|^2\,d\sigma_\epsilon(\mathbf{x}), \notag
\end{align}
where the constant $c$ is the $V$-elliptic constant for the solid elasticity tensor $a_{ijkl}$ defined in \eqref{a-coercive}.  By the extension result in Theorem~\ref{thm:ext1}, there exist operators $T_\epsilon^f$ and $T_\epsilon^s$ that extend $\mathbf{w}$ to $\Omega_1$ from $\Omega_\epsilon^f$ and $\Omega_\epsilon^s$, respectively, such that the following estimates are valid with positive constants $C_f$ and $C_s$, independent of $\epsilon$:
\begin{align*}
&\int_{\Omega_1}E_{ij}(T_\epsilon^f\mathbf{u})\overline{E_{ij}(T_\epsilon^f\mathbf{u})}\,d\mathbf{x}\leq C_f\int_{\Omega_\epsilon^f}E_{ij}(\mathbf{u})\overline{E_{ij}(\mathbf{u})}\,d\mathbf{x}\hspace{3mm}\forall \mathbf{u}\in \mathbf{V}^f,\\
&\int_{\Omega_1}E_{ij}(T_\epsilon^s\mathbf{u})\overline{E_{ij}(T_\epsilon^s\mathbf{u})}\,d\mathbf{x}\leq C_s\int_{\Omega_\epsilon^s}E_{ij}(\mathbf{u})\overline{E_{ij}(\mathbf{u})}\,d\mathbf{x}\hspace{3mm}\forall \mathbf{u}\in \mathbf{V}^s,
\end{align*}
where $V^{s,f}$ and $\Omega_1$ are defined in \eqref{thm_ext_V} and \eqref{thm_ext_omega1}, respectively.  
Since the extended functions belong to $H^1_0(\Omega_1)$, Korn's inequality implies that:
\begin{equation*}
\int_{\Omega_\epsilon^s\times\Omega_\epsilon^f}E_{ij}(\mathbf{w})\overline{E_{ij}(\mathbf{w})}d\mathbf{x}\quad \ge D' (\|\mathbf{w}\|_{H^1(\Omega_\epsilon^f)}^2+\|\mathbf{w}\|_{H^1(\Omega_\epsilon^s)}^2)\notag
\end{equation*}
where the positive constant $D'$ depends only on $C_s$, $C_f$ and the Korn's constant of $\Omega_1$. Finally, \eqref{coersive-one} becomes:
\begin{equation*}
c^\epsilon(\mathbf{w},\mathbf{w})+b^\epsilon\left(\mathbf{w},\mathbf{w}\right) {+\int_{\Gamma_\epsilon}\left|\llbracket\mathbf{w}\rrbracket_s^f\right|^2\,d\sigma_\epsilon(\mathbf{x})}\ge \min\left(\frac{cD'}{C_s},\frac{\mu D'}{C_f},1\right)\|\mathbf{w}\|_V^2.
\end{equation*}
\end{proof}
With these lemmas, Theorem \ref{uniform_bound} can be proved as follows.
\begin{proof}[Proof of Theorem \ref{uniform_bound}]
	By setting $\mathbf{\hat{w}}=\mathbf{\hat{u}_\epsilon}$ in (\ref{LVF}), we obtain:
\begin{align*}
{\rm Re}\left(\frac{1}{\lambda}\right)\int_{\Omega}\hat{f}^i\,\overline{\hat{u}_\epsilon^i}\,d\mathbf{x}=&{\rm Re}\left(\lambda\right)\left[\int_{\Omega_\epsilon^s}\rho^s \hat{u}_\epsilon^i\,\overline{\hat{u}_\epsilon^i}\,d\mathbf{x}+\int_{\Omega_\epsilon^f}\rho^f\hat{u}_\epsilon^i\,\overline{\hat{u}_\epsilon^i}\,d\mathbf{x}\right]+\epsilon^2b^\epsilon\left(\mathbf{\hat{u}_\epsilon},\mathbf{\hat{u}_\epsilon}\right)\\
	&+{\rm Re}\left(\frac{1}{\lambda}\right)c^{\epsilon}(\mathbf{\hat{u}_\epsilon},\mathbf{\hat{u}_\epsilon})+\epsilon\int_{\Gamma_\epsilon}\alpha\llbracket\mathbf{\hat{u}_\epsilon}\rrbracket_s^f\cdot \overline{\llbracket\mathbf{\hat{u}_\epsilon}\rrbracket}_s^f\,d\sigma_\epsilon(\mathbf{x})
\end{align*}

Besides, from Lemma \ref{time_bound}, we can easily conclude that, for $Re(\lambda)>r$, we have:
\[
\|\mathbf{\hat{u}}_\epsilon\|_{L^2(\Omega)}^2\le \int_0^\infty e^{-2\lambda t} \| \mathbf{u}_\epsilon\|_{L^2(\Omega)}^2 dt\le C\int_0^\infty e^{-2(\lambda-r) t} dt=\frac{C}{2(Re(\lambda)-r)}\,,
\]
i.e. $\|\mathbf{\hat{u}}_\epsilon\|_{L^2(\Omega)}^2$ is uniformly bounded with respect to $\epsilon$. 
Therefore, for ${\rm Re}(\lambda)>r$, by taking into account \eqref{f}, we have the following bounds:
\begin{align}
	c^{\epsilon}(\mathbf{\hat{u}_\epsilon},\mathbf{\hat{u}_\epsilon})&\leq C \hspace{3mm}\forall\epsilon,
	\label{unifboundc}\\
	\epsilon^2b^\epsilon\left(\mathbf{\hat{u}_\epsilon},\mathbf{\hat{u}_\epsilon}\right)&\leq C\hspace{3mm}\forall\epsilon,	\label{unifboundb}\\
	\epsilon\norm{\llbracket\mathbf{\hat{u}_\epsilon}\rrbracket_s^f}_{L^2(\Gamma_\epsilon)}&\leq C\hspace{3mm}\forall\epsilon,	\label{unifboundgamma}
\end{align}
where, for simplicity, we write $\mathbf{\hat{u}_\epsilon}$ instead of $\mathbf{\hat{u}_\epsilon}(\lambda)$. 
The uniform bound \eqref{unifboundu} is then implied by Lemma \ref{1-coercive}, (\ref{unifboundc}), (\ref{unifboundb}), and (\ref{unifboundgamma}). The uniform bound \eqref{unifbounddivu} is a direct consequence of (\ref{weakLap}). To show the uniform bound of the gradient restricted to the solid phase \eqref{uni_deri_solid}, note that (\ref{c-eps}), (\ref{unifboundc}), and (\ref{a-coercive}) lead to
\begin{align*}
&c\int_{\Omega_\epsilon^s}E_{ij}(\mathbf{\hat{u}_\epsilon})\overline{E_{ij}(\mathbf{\hat{u}_\epsilon})}\,d\mathbf{x}\\
&\quad\leq c^{\epsilon}(\mathbf{\hat{u}_\epsilon},\mathbf{\hat{u}_\epsilon})=\int_{\Omega_\epsilon^s}a^s_{ijkl}E_{kl}(\mathbf{\hat{u}_\epsilon})\overline{E_{ij}(\mathbf{\hat{u}_\epsilon})}\,d\mathbf{x}+\int_{\Omega_\epsilon^f}\gamma\left|{\rm div}\mathbf{\hat{u}_\epsilon}\right|^2\,d\mathbf{x}\leq C.
\end{align*}

From Korn's inequality for $H_0^1(\Omega_1)$, Theorem~\ref{thm:ext1}, and the inequality above, we have:
\begin{align*}
\int_{\Omega_\epsilon^s}\left|\frac{\partial \hat{u}_\epsilon^i}{\partial x_j}\right|^2d\mathbf{x}&\leq\int_{\Omega_1}\frac{\partial T_\epsilon\hat{u}_\epsilon^i}{\partial x_j}\overline{\frac{\partial T_\epsilon\hat{u}_\epsilon^i}{\partial x_j}}\,d\mathbf{x}\\
	&\leq C(\Omega_1)\int_{\Omega_1}E_{ij}(\mathbf{T_\epsilon\hat{u}_\epsilon})\overline{E_{ij}(T_\epsilon\mathbf{\hat{u}_\epsilon})}\,d\mathbf{x}\\
	&\leq C\int_{\Omega_\epsilon^s}E_{ij}(\mathbf{\hat{u}_\epsilon})\overline{E_{ij}(\mathbf{\hat{u}_\epsilon})}\,d\mathbf{x}\leq C.
\end{align*}

\end{proof}

{With the bounds in Theorem \ref{uniform_bound}, the sequence $\{\mathbf{u}_\epsilon\}$ can be analyzed by using the compactness theorems of the two-scale convergence. {In these bounds, notice that $\{\mathbf{u}_\epsilon\}$ as a whole are uniformly bounded in the $E_0$-norm but not in the $H_1$ norm. On the other hand, the restriction of $\{\mathbf{u}_\epsilon\}$ in the solid phase is uniformly bounded in the $H_1$ norm. As we will see in the next section, this will result in different convergence behaviors in the solid phase and in the fluid phase.}  
\section{Two-scale limits}
{The section is devoted to developing various two-scale limits of $\{\mathbf{u}_\epsilon\}$} and the relations between them. We first note that the bounds \eqref{unifbounduep} and \eqref{unifboundu} imply the following lemma.}
\begin{lemma}
\label{lem:2scalecvuep}
	We can extract a subsequence of $\{\hat{\mathbf{u}}_\epsilon\}$ such that: 
\begin{align}
	\int_{\Omega}\hat{u}_\epsilon^k\psi^\epsilon\phi\,d\mathbf{x}&\rightarrow\int_{\Omega\times Y}w_o^k(\mathbf{x},\mathbf{y})\psi(\mathbf{y})\phi(\mathbf{x})\,d\mathbf{x}\,d\mathbf{y},\hspace{4mm}1\leq k\leq 3,\label{2scluep}\\
	\int_{\Omega}\epsilon\frac{\partial \hat{u}_\epsilon^k}{\partial x_l}\psi^\epsilon\phi\,d\mathbf{x}&\rightarrow\int_{\Omega\times Y}\frac{\partial w_o^k}{\partial y_l}(\mathbf{x},\mathbf{y})\psi(\mathbf{y})\phi(\mathbf{x})\,d\mathbf{x}\,d\mathbf{y},\hspace{4mm}1\leq k,l\leq 3,\label{2sclduep}
\end{align}
for all $\psi\in L^2_p$, $\phi\in\mathscr{K}(\overline{\Omega})$, where:
\begin{equation*}
	\mathbf{w_o(x,y)}=(w_o^k)\in \mathbf{L^2}(\Omega;\mathbf{H_p^1}(Y_s\cup Y_f)),
\end{equation*}
\begin{equation}
	\label{divwo}
	{\rm div}_\mathbf{y}\mathbf{w_o(x,y)}=0.
\end{equation}

Moreover, {this two-scale limit $w_0$ is related to the $E_0$-limit $\hat{\mathbf{u_0}}$ (\ref{weakLap}) as follows}
\begin{equation}
	\label{uowo}
	\mathbf{\hat{u}_0}=\langle\mathbf{w_0}\rangle(\mathbf{x})
\end{equation}
\end{lemma}
\begin{proof}
	Since $\left\{\hat{u}_\epsilon^k\right\}_{\epsilon>0}$ is bounded in $\mathbf{L^2}(\Omega)$, (\ref{2scluep}) follows immediately by Lemma~\ref{2scaleth}.  Property (\ref{2sclduep}) follows as a consequence of (\ref{unifboundu}), (\ref{2scluep}), Remark~\ref{th:Nguetseng89}, with an integration by parts argument similar to the one used in Proposition 1.14 in \cite{Allaire1992}.  From (\ref{unifbounddivu}) and (\ref{2sclduep}), we have as $\epsilon\rightarrow0$, taking $k=l$:
\begin{align*}
	\epsilon\int_{\Omega}{\rm div}_\mathbf{x}\mathbf{\hat{u}_\epsilon}\psi^\epsilon\phi\,d\mathbf{x}&\rightarrow0 \hspace{1cm}\\
	\epsilon\int_{\Omega}{\rm div}_\mathbf{x}\mathbf{\hat{u}_\epsilon}\psi^\epsilon\phi\,d\mathbf{x}&\rightarrow\int_{\Omega\times Y}{\rm div}_\mathbf{y}\mathbf{w_0}\psi(\mathbf{y})\phi(\mathbf{x})\,d\mathbf{x}\,d\mathbf{y}, 
\end{align*}
from where we obtain (\ref{divwo}).  As for (\ref{uowo}), it follows from (\ref{weakLap}) and (\ref{2scluep}).
\end{proof}
Because of the uniform boundedness of the gradient in the solid phase \eqref{uni_deri_solid}, more can be said about the two-scale limit of $\hat{\mathbf{u}}_\epsilon$ as follows.
\begin{lemma}
\label{lem:due2scsol}
A subsequence can be extracted from the one in Lemma~\ref{lem:2scalecvuep}, such that:
\begin{equation}
	\label{pue2scvs}
\int_{\Omega_\epsilon^s}\frac{\partial \hat{u}_\epsilon^k}{\partial x_l}\psi^\epsilon\phi\,d\mathbf{x}\rightarrow\int_{\Omega\times Y_s}\left[\frac{\partial u^k}{\partial x_l}(\mathbf{x})+\frac{\partial u^k_1}{\partial y_l}(\mathbf{x},\mathbf{y})\right]\psi(\mathbf{y})\phi(\mathbf{x})\,d\mathbf{x}\,d\mathbf{y},
\end{equation}
for $1\leq k,l\leq3$; for all $\psi\in\mathbf{L^2_p}$ and all $\phi\in\mathscr{K}(\overline{\Omega})$, where $\mathbf{u}=\left\{u^k\right\}\in\mathbf{H^1_0}(\Omega)$, $\mathbf{u_1}=\left\{u_1^k\right\}\in\mathbf{L^2}(\Omega;\mathbf{H^1_p}(Y_s)/\mathbb{C}^3)$.

Moreover, the limit $\mathbf{w_o}$ in Lemma~\ref{lem:2scalecvuep} decomposes as follows:
\begin{equation}
\label{w0-decomp}
	\mathbf{w_0}(\mathbf{x},\mathbf{y})=\mathbf{u}(\mathbf{x})+\mathbf{u_r}(\mathbf{x},\mathbf{y})
\end{equation}
with $\mathbf{u_r}\in\mathbf{L^2}(\Omega;\mathbf{H^1_p}(Y_f\cup Y_s))$, $\mathbf{u_r}(\mathbf{x},\mathbf{y})=0$ for $\mathbf{y}\in Y_s$  and ${\rm div}_\mathbf{y}\mathbf{u_r}=0$, { i.e., $\mathbf{u_r}\in L^2(\Omega,W)$ with $W$ defined in \eqref{spaceW}.}
\end{lemma}
\begin{proof}
This follows from using \eqref{uni_deri_solid} and applying Lemma~\ref{lem:2scalecvuep} and Theorem~\ref{2scaleth} by letting $Y_o=Y_s$.
\end{proof}
{Note that the uniform bound on the gradient in the solid phase guarantees the decomposition \eqref{w0-decomp} of $\mathbf{w_0}$, which is the two scale limit of $\mathbf{u}_\epsilon$. On the other hand, the divergence of $\mathbf{u}_\epsilon$ is uniformly bounded in both phases. Hence it is natural to study how the two scale limit of $\{ \rm{div}\hat{\mathbf{u}}_\epsilon\}$ is related to $\mathbf{u}_1$; this is the subject of Lemma \ref{lem:pe2sv}. In preparation for stating this lemma, we recall the definition of} the acoustic pressure $\hat{p}_\epsilon$ 
\[
\hat{p}_\epsilon:=-\gamma\,{\rm div}\mathbf{\hat{u}_\epsilon} \mbox{ in } \Omega_\epsilon^f, \mbox{ with } \gamma:=c_0^2 \rho^f.
\]
Note that $\hat{p}_\epsilon$ satisfies $\hat{p}_\epsilon\in L^2(\Omega_\epsilon^f)$, with $\norm{\hat{p}_\epsilon}_{L^2(\Omega_\epsilon^f)}\leq C$, for all $\epsilon>0$.  Consider $D_\epsilon^{kl}(\mathbf{x}) :=\chi_s(\mathbf{x}) \displaystyle\frac{\partial \hat{u}_\epsilon^k}{\partial x_l} -\chi_f (\mathbf{x}) \frac{\delta_{kl}}{3\gamma}\hat{p}_\epsilon \in L^2(\Omega)$, $1\leq k,l\leq3$.  Then we have:
\begin{equation*}
\int_{\Omega}D_\epsilon^{kl}\,v\,d\mathbf{x}=\int_{\Omega_\epsilon^s}\frac{\partial \hat{u}_\epsilon^k}{\partial x_l}\,v\,d\mathbf{x}-\frac{\delta_{kl}}{3\gamma}\int_{\Omega_\epsilon^f}\hat{p}_\epsilon\,v\,d\mathbf{x},
\end{equation*}
for all $v\in \mathcal{K}(\overline{\Omega})$.  Since $\norm{D_\epsilon^{kl}}_{L^2(\Omega)}\leq C$, for all $\epsilon>0$ and $1\leq k,l\leq3$, the sequence $D_\epsilon^{kl}$ has a weak limit in the sense of Lemma~\ref{2scaleth}, which we denote by $D^{kl}$.  Taking $\mathbf{w^\epsilon}=1$ and $\phi=v$ in Lemma~\ref{2scaleth}, we obtain:
\begin{align*}
	\int_{\Omega}D_\epsilon^{kl}\,v\,d\mathbf{x}\rightarrow\int_{\Omega\times Y}D^{kl}(\mathbf{x},\mathbf{y})v(\mathbf{x})\,d\mathbf{x}\,d\mathbf{y},
\end{align*}
and, by Lemma~\ref{lem:due2scsol}, we can conclude:
\begin{equation}
	\label{dkh}
	D^{kl}(\mathbf{x},\mathbf{y})=\frac{\partial u^k}{\partial x_l}(\mathbf{x})+\frac{\partial u^k_1}{\partial y_l}(\mathbf{x},\mathbf{y}),\hspace{4mm} \text{ for $(\mathbf{x},\mathbf{y})\in\Omega\times{Y}_s$.}
\end{equation}

	Letting $p_0(\mathbf{x},\mathbf{y}):=-\gamma D^{kk}(\mathbf{x},\mathbf{y})$, for $(\mathbf{x},\mathbf{y})\in \Omega\times Y_f$, we are ready to state the following lemma. 
\begin{lemma}
\label{lem:pe2sv}
As $\epsilon\downarrow0$ ($\epsilon$ a subsequence from the one in Lemma~\ref{lem:due2scsol}), for all $\psi\in L^2_p$, all $\phi\in\mathscr{K}(\overline{\Omega})$, the acoustic pressure $\hat{p}_\epsilon:=-\gamma\,{\rm div}\hat{\mathbf{u}_\epsilon}$ two-scale converges as follows
\begin{equation*}
	\int_{\Omega_\epsilon^f}\hat{p}_\epsilon\psi^\epsilon\phi\,d\mathbf{x}\rightarrow\int_{\Omega\times Y_f}p_0(\mathbf{x},\mathbf{y})
\psi(\mathbf{y})\phi(\mathbf{x})\,d\mathbf{x}\,d\mathbf{y}, \hspace{3mm}p_0\in L^2(\Omega;L^2_p(Y_f)).
\end{equation*}
Moreover, $\mathbf{u}$ and $\mathbf{u_1}$ in \eqref{pue2scvs} and the two-scale limit of $\rm{div} \hat{\mathbf{u}}_\epsilon$ and $\mathbf{u_r}$ in \eqref  {w0-decomp} satisfy the relation:
\begin{equation}
\label{s-f-rel}
	\int_{Y_s}{\rm div}_\mathbf{y}\mathbf{u_1}(\mathbf{x},\mathbf{y})\,d\mathbf{y}=\left|Y_f\right|{\rm div}\mathbf{u}((\mathbf{x})+{\rm div}\int_{Y_f}\mathbf{u_r}((\mathbf{x},\mathbf{y})\,d\mathbf{y}+\frac{1}{\gamma}\int_{Y_f}p_o((\mathbf{x},\mathbf{y})\,d\mathbf{y}.
\end{equation}
\end{lemma}
\begin{proof}
The lemma follows from (\ref{dkh}) and (\ref{w0-decomp})).  If $k=l$, we have:
\begin{align*}
	\int_{\Omega\times Y}D^{kk}(\mathbf{x},\mathbf{y})\psi(\mathbf{y})\phi(\mathbf{x})\,d\mathbf{x}\,d\mathbf{y}&=\int_{\Omega\times Y_s}\left[{\rm div}\mathbf{u}(\mathbf{x})+{\rm div}_\mathbf{y}\mathbf{u_1}(\mathbf{x},\mathbf{y})\right]\psi(\mathbf{y})\phi(\mathbf{x})\,d\mathbf{x}\,d\mathbf{y}\\
	&\quad-\frac{1}{\gamma}\int_{\Omega\times Y_f}p_0(\mathbf{x},\mathbf{y})\psi(\mathbf{y})\phi(\mathbf{x})\,d\mathbf{x}\,d\mathbf{y}.
\end{align*}
To obtain (\ref{s-f-rel}), by (\ref{weakLap}) we have, for all $v\in\mathscr{D}(\Omega)$, that:
\begin{align*}
	\int_{\Omega}{\rm div}\,\mathbf{\hat{u}_\epsilon}\,v\,d\mathbf{x}&\rightarrow\int_{\Omega}{\rm div}\,\mathbf{\hat{u}_o}\,v\,d\mathbf{x}\\
	&=\int_{\Omega\times Y}{\rm div}\mathbf{u}(\mathbf{x})\,v\,d\mathbf{x}\,d\mathbf{y}+\int_{\Omega\times Y_f}{\rm div}\mathbf{u_r}(\mathbf{x},\mathbf{y})\,v\,d\mathbf{x}\,d\mathbf{y}.
\end{align*}
On the other hand, we have:
\begin{align*}
	\int_{\Omega}{\rm div}\,\mathbf{\hat{u}_\epsilon}\,v\,d\mathbf{x}&\rightarrow\int_{\Omega\times Y_s}{\rm div}\, \mathbf{u}(\mathbf{x})\,v\,d\mathbf{x}\,d\mathbf{y}\\
	&\quad+\int_{\Omega\times Y_s}{\rm div}_\mathbf{y}\, \mathbf{u_1}(\mathbf{x},\mathbf{y})\,v\,d\mathbf{x}\,d\mathbf{y}-\frac{1}{\gamma}\int_{\Omega\times Y_f}p_0(\mathbf{x},\mathbf{y})\,v\,d\mathbf{x}\,d\mathbf{y}.
\end{align*}
Hence, we obtain: 
\begin{align*}
	&\int_{\Omega\times Y_f}{\rm div}\mathbf{u}(\mathbf{x})\,v\,d\mathbf{x}\,d\mathbf{y}+\int_{\Omega\times Y_f}{\rm div}\mathbf{u_r}(\mathbf{x},\mathbf{y})\,v\,d\mathbf{x}\,d\mathbf{y}\\
	&=\int_{\Omega\times Y_s}{\rm div}_\mathbf{y}\, \mathbf{u_1}(\mathbf{x},\mathbf{y})\,v\,d\mathbf{x}\,d\mathbf{y}-\frac{1}{\gamma}\int_{\Omega\times Y_f}p_0(\mathbf{x},\mathbf{y})\,v\,d\mathbf{x}\,d\mathbf{y}.
\end{align*}
Therefore $\mathbf{u_1}$ and $\mathbf{u_r}$ satisfy the relation described by \eqref{s-f-rel}.
\end{proof}

\begin{remark}
	In the sequel, $\epsilon$ represents the subsequence involved in Lemma~\ref{lem:pe2sv}.  Observe that Lemmas~\ref{lem:2scalecvuep}-\ref{lem:pe2sv} hold simultaneously for that subsequence.
\end{remark}

\section{Derivation of the local problems.}
\label{sec:derivlocprob}  
{In the previous section, we have shown that in the solid phase, the two-scale limit $\mathbf{w_0}(\mathbf{x},\mathbf{y})$ is exactly the  $\mathbf{u}(\mathbf{x})$ in \eqref{w0-decomp}, whereas in the fluid phase, it is $\mathbf{u}(\mathbf{x})+\mathbf{u_r}(\mathbf{x},\mathbf{y})$. Also shown in the previous section is that for the solid phase, the gradient of $\{\hat{\mathbf{u}}_\epsilon\}$ two-scale converges to $\nabla\mathbf{u}+\nabla_y\mathbf{u}_1$ while in the fluid phase, it can only be concluded that the acoustic pressure $\{\hat{p}_\epsilon\}$ two-scale converges to $p_0$. Moreover, the two-scale limit $\mathbf{w_0}$ is related to the $E_0$-limit $\mathbf{\hat{u}_0}$ by \eqref{uowo} and $<\mathbf{w_0}>(\mathbf{x})=\mathbf{u(\mathbf{x})}+<\mathbf{u_r}>(\mathbf{x})$. In this section, we will first prove that $p
_0$ does not depend on $\mathbf{y}$. 
}

{The focus in this section is on the corrector term $\mathbf{u_1}$ of the gradient in the solid and the corrector term $\mathbf{u_r}$ for the fluid, given $\mathbf{u}(\mathbf{x})$ and $p_0(\mathbf{x})$. We first summarize the main results in the following theorems.}

\begin{theorem}[Local problem for $\mathbf{u}_1$]
{\label{sec:probu1} } 
The limit $p_0$ does not depend on $\mathbf{y}$. Furthermore, the local problem for  $\mathbf{u}_1$ is as follows
\begin{align}
\label{locprou1-2}
& \mbox{Find }\mathbf{u_1}\in \mathbf{H^1_p}(Y_s)/\mathbb{C}^3\mbox{ such that }\notag\\
&q\left(\mathbf{u_1}(\mathbf{x},\cdot),\mathbf{w}\right)=-\frac{\partial u^k}{\partial x_l}(\mathbf{x})\int_{Y_s}a_{ijkl}\overline{\frac{\partial w^i}{\partial y_j}}\,d\mathbf{y}-p_0(\mathbf{x})\int_{Y_s}\overline{{\rm div}_\mathbf{y}\mathbf{w}}\,d\mathbf{y}\\
&\forall \mathbf{w}\in\mathbf{H^1_p}(Y_s)/\mathbb{C}^3,\notag
\end{align}
where $q(\cdot,\cdot)$ represents the sesquilinear form given by:
\begin{equation}
\label{a}
q(\mathbf{v},\mathbf{w})=\int_{Y_s}a_{ijkl}\frac{\partial v^k}{\partial y_l}\overline{\frac{\partial w^i}{\partial y_j}}\,d\mathbf{y}=\int_{Y_s}a_{ijkl}  e_{ij}(\mathbf{v}) e_{kl}(\overline{\mathbf{w}})\,d\mathbf{y}.
\end{equation}
 This problem is uniquely solvable.
\end{theorem}
\begin{proof}
We start by testing problem (\ref{LVF}) with $\mathbf{w}=\epsilon(\mathbf{w_s^\epsilon}+\mathbf{w_f^\epsilon})\phi$, where  $\mathbf{w_s}\in\mathbf{H^1_p(Y_s)}$, $\mathbf{w_f}\in\mathbf{H^1_p}(Y_f)$, $\mathbf{w_s}(\mathbf{y})=0$ for $\mathbf{y}\in Y_f$, $\mathbf{w_f}(\mathbf{y})=0$ for $\mathbf{y}\in Y_s$, $\mathbf{w_s}\cdot\mathbf{n}=\mathbf{w_f}\cdot\mathbf{n}$ on $\Gamma$ and $\phi\in\mathscr{D}(\Omega)$ to obtain: 
\begin{align*}
	&\epsilon\int_{\Omega_\epsilon^s}\hat{f}^i\overline{{w_s^\epsilon}^i\phi}d\mathbf{x}+\epsilon\int_{\Omega_\epsilon^f}\hat{f}^i\overline{{w_f^\epsilon}^i\phi}\,d\mathbf{x}\notag\\
	&=\lambda^2\,\epsilon\,\int_{\Omega_\epsilon^s}\rho^s \hat{u}_\epsilon^i\overline{{w_s^\epsilon}^i\phi}d\mathbf{x}+\lambda^2\,\epsilon\,\int_{\Omega_\epsilon^f}\rho^f\hat{u}_\epsilon^i\overline{{w_f^\epsilon}^i\phi}d\mathbf{x}+ c^{\epsilon}(\mathbf{\hat{u}_\epsilon},\epsilon(\mathbf{w_s^\epsilon}+\mathbf{w_f^\epsilon})\phi)\notag\\
	&\quad+\lambda\epsilon^2b^\epsilon\left(\mathbf{\hat{u}_\epsilon},\mathbf{\epsilon w_f^\epsilon\phi}\right)+\lambda\,\epsilon^2\int_{\Gamma_\epsilon}\alpha\llbracket\mathbf{\hat{u}_\epsilon}\rrbracket_s^f\cdot \overline{\llbracket(\mathbf{w_s^\epsilon}+\mathbf{w_f^\epsilon})\phi\rrbracket}_s^f\,d\sigma_\epsilon(\mathbf{x}).
\end{align*}

Observe that, as $\epsilon\downarrow0$, every term goes to $0$ except for $c^{\epsilon}(\mathbf{\hat{u}_\epsilon},\epsilon(\mathbf{w_s^\epsilon}+\mathbf{w_f^\epsilon})\phi)$.  We study this term in detail:
\begin{align*}
c^{\epsilon}(\mathbf{\hat{u}_\epsilon},\epsilon(\mathbf{w_s^\epsilon}+\mathbf{w_f^\epsilon})\phi)
&=\epsilon\int_{\Omega_\epsilon^s}a_{ijkl}\frac{\partial \hat{u}_\epsilon^k}{\partial x_l}\overline{\left({w_s^\epsilon}^i\frac{\partial \phi}{\partial x_j}\right)}d\mathbf{x}+\int_{\Omega_\epsilon^s}a_{ijkl}\frac{\partial \hat{u}_\epsilon^k}{\partial x_l}\overline{\left(\phi \left(\frac{\partial w_s^i}{\partial y_j}\right)^\epsilon\right)}d\mathbf{x}\\
&\quad+\epsilon\int_{\Omega_\epsilon^f}\gamma({\rm div} \mathbf{\hat{u}_\epsilon})\overline{\left(\nabla\phi\cdot\mathbf{w_f^\epsilon}\right)}d\mathbf{x}+\int_{\Omega_\epsilon^f}\gamma({\rm div} \mathbf{\hat{u}_\epsilon})\overline{\left(\phi\left({\rm div}_y\mathbf{w_f}\right)^\epsilon \right)}d\mathbf{x}.
\end{align*}
Note that the first and the third terms in the previous expression go to $0$ as $\epsilon\downarrow0$.  By Lemma~\ref{lem:due2scsol} (with $\psi=a_{ijkl}\overline{\left(\partial w_s^i/\partial y_j\right)}$) and by Lemma~\ref{lem:pe2sv} (with $\psi=\overline{{\rm div}_\mathbf{y}\mathbf{w_f}}$), we obtain the local problem for $\mathbf{u_1}$:
\begin{equation}
\label{locprou1}
\int_{Y_s}a_{ijkl}\left[\frac{\partial u^k}{\partial x_l}(\mathbf{x})+\frac{\partial u_1^k}{\partial y_l}(\mathbf{x},\mathbf{y})\right]\overline{\frac{\partial w_s^i}{\partial y_j}}(\mathbf{y})\,d\mathbf{y}-\int_{Y_f}p_0(\mathbf{x},\mathbf{y})\overline{{\rm div}_\mathbf{y}\mathbf{w_f}(\mathbf{y})}\,d\mathbf{y}=0,
\end{equation}
for all $\mathbf{w_s}\in\mathbf{H^1_p}(Y_s)$, $\mathbf{w_f}\in\mathbf{H^1_p}(Y_f)$ with $\mathbf{w_s}(\mathbf{y})=0$ for $\mathbf{y}\in Y_f$, $\mathbf{w_f}(\mathbf{y})=0$ for $\mathbf{y}\in Y_s$ and $\mathbf{w_s}\cdot\mathbf{n}=\mathbf{w_f}\cdot\mathbf{n}$ on $\Gamma$.

By choosing $\mathbf{w_s=0}$, (\ref{locprou1}) becomes:
\begin{equation*}
-\int_{Y_f}p_0(\mathbf{x},\mathbf{y})\overline{{\rm div}_\mathbf{y}\mathbf{w}(\mathbf{y})}\,d\mathbf{y}=0,
\end{equation*}
for all $\mathbf{w}\in\mathbf{H^1_p}(Y_f)$ with $ \mathbf{w}\cdot \mathbf{n}=0$ on $\partial Y_f$.  Therefore, it can be concluded that $\displaystyle\int_{Y_f}\nabla_\mathbf{y}p_0(\mathbf{x},\mathbf{y})\cdot\overline{\mathbf{w}}\,d\mathbf{y}=0$ by a density argument and integration by parts. Hence $p_0$ does not depend on $\mathbf{y}$.  In other words, $p_0\in L^2(\Omega)$.  
  To further simplify \eqref{locprou1}, we observe that, for all $\mathbf{w}\in V_Y$ such that $\mathbf{w}\cdot \mathbf{n}$ is continuous across $\Gamma$, we have:
\begin{equation*}
\int_{Y_f}p_0(\mathbf{x})\overline{{\rm div}_\mathbf{y}\mathbf{w}(\mathbf{y})}\,d\mathbf{y}=p_0(\mathbf{x})\int_{\partial Y_f}\overline{\mathbf{w}}\cdot\mathbf{n}d\sigma(\mathbf{y})=-p_0(\mathbf{x})\int_{Y_s}\overline{{\rm div}_\mathbf{y}\mathbf{w}(\mathbf{y})}\,d\mathbf{y}.
\end{equation*}
This leads to \eqref{locprou1-2}. To prove the uniqueness, we need to check that $q(\cdot,\cdot)$ is coercive on $\mathbf{H^1_p}(Y_s)/\mathbb{C}^3$, which means that there exists $c>0$ such that:
\begin{equation}
	\label{a-coerc}
	q(\mathbf{w},\mathbf{w})\geq C\norm{\mathbf{w}}^2_{\mathbf{H^1_p}(Y_s)/\mathbb{C}^3},\hspace{4mm}\forall \mathbf{w}\in\mathbf{H^1_p}(Y_s)/\mathbb{C}^3.
\end{equation} 
 But (\ref{a-coerc}) follows from (\ref{symmetry}), (\ref{a-coercive}), and an application of Korn's inequality for $\mathbf{H^1_p}(Y_s)/\mathbb{C}^3$.
 \end{proof}
\begin{theorem}[The local problem for $\mathbf{u_r}$]
\label{sec:probur}
The local problem for $\mathbf{u_r}$ is 
\begin{align}
\label{locprob-ur}
&\lambda^2\rho^f\int_{Y_f}u_r^i(\mathbf{x},\mathbf{y})\overline{w^i(\mathbf{y})}\,d\mathbf{y}+2\lambda\mu\int_{Y_f}\frac{\partial u_r^i}{\partial y_j}(\mathbf{x},\mathbf{y})\overline{\frac{\partial w^i}{\partial y_j}}\,d\mathbf{y}+\lambda\alpha\int_\Gamma u_r^i(\mathbf{x},\mathbf{y})\overline{w^i}\,d\sigma(\mathbf{y})\\
&=\left(\hat{f}^i(\mathbf{x})-\lambda^2\rho^fu^i(\mathbf{x})-\frac{\partial p_0}{\partial x_i}(\mathbf{x})\right)\int_{Y_f}\overline{w^i}\,d\mathbf{y},\hspace{3mm}\forall\mathbf{w}\in W.\notag
\end{align}
The above problem is coercive in the $V_Y$ norm and, hence, has a unique solution. Note that It is the weak formulation of the cell problem:
\begin{equation}
\label{cell_ur}
\begin{cases}
  \displaystyle\lambda^2\rho^f u_r^i(\mathbf{x}, \mathbf{y})+ 2\lambda\mu \frac{\partial e_{ij}(\mathbf{u_r})}{\partial y_j} =\left(\hat{f}^i(\mathbf{x})-\lambda^2\rho^f u^i(\mathbf{x})-\frac{\partial p_0}{\partial x_i}(\mathbf{x})\right) &\text{ in } Y_f ,\\
  \displaystyle 2 \mu  e_{ij}(\mathbf{u_r}) n^j=\alpha u_r^i &\text{ on }\Gamma, \hspace{2mm} \forall\mathbf{w}\in W.
\end{cases}
\end{equation}
\end{theorem}
{As can be seen in the theorem above, the interface term resulting from the slip condition is part of the local problem for  $\mathbf{u_r}$. The following lemma is hence necessary in proving Theorem \ref{sec:probur} so we state it here. Note that because of the discontinuity on the interface, we cannot directly apply Proposition 2.6 in  \cite{Allaire95}. Instead, we generalized that proposition to the following lemma. The main point  is to show that the two-scale convergence limit on $\Gamma_\epsilon$  in the sense of \eqref{2scvboud} is indeed the trace of the two-scale convergence limit, cf. Definition \ref{2cale-conv} for our case.} 
\begin{lemma}
\label{l:twoscaleconvinterf} A subsequence can be extracted from the sequence in Lemma~\ref{lem:2scalecvuep} such that the following convergence holds.
\begin{align}
&\lambda\,\epsilon\,\alpha\int_{\Gamma_\epsilon^f}\mathbf{\hat{u}_\epsilon}\cdot \overline{\llbracket\mathbf{w^\epsilon}\phi\rrbracket}_s^f\,d\sigma_\epsilon(\mathbf{x})-\lambda\,\epsilon\alpha\int_{\Gamma_\epsilon^s}\mathbf{\hat{u}_\epsilon}\cdot \overline{\llbracket\mathbf{w^\epsilon}\phi\rrbracket}_s^f\,d\sigma_\epsilon(\mathbf{x})\notag\\
&\rightarrow \lambda\,\alpha\int_{\Omega}\int_{\Gamma}\mathbf{u_r}(\mathbf{x})\overline{\phi(\mathbf{x},\mathbf{y})}\,\overline{\llbracket\mathbf{w}(\mathbf{y})\rrbracket}_s^f\,d\sigma(\mathbf{y}) d\mathbf{x}\,, \mbox{  \hspace{0.2in}as   } \epsilon \to 0.\label{lab_4}
\end{align}
\end{lemma}
\begin{proof}
For the solid part, we consider $\widetilde {\mathbf{\hat{u}_\epsilon}^s }$, the extension by zero of $\mathbf{\hat{u}_\epsilon}^s:=\mathbf{\hat{u}_\epsilon}\Big|_{\Omega_\epsilon^s}$, which coinsides with $\mathbf{\hat{u}_\epsilon}$ in ${\Omega_\epsilon^s}$.  Then, there exist $\mathbf{u}(x)\in H_0^1(\Omega)$ and $\mathbf{u_1}(x,y)\in L^2(\Omega,H_p^1(Y_s)/\mathbb{R})$ such that $\widetilde{\mathbf{\hat{u}}_\epsilon^s}$ two-scale converges to $\mathbf{u}(x)\chi(Y_s)$ and $\nabla\widetilde{\mathbf{\hat{u}}_\epsilon^s}$ two-scale converges to $(\nabla\mathbf{u}(x)+\nabla_y\mathbf{u_1}(x,y))\chi(Y_s)$, cf. Theorem~2.9 in \cite{Allaire1992}.    

The uniform boundedness of the interface integral in \eqref{unifboundgamma} is still one order shy of the assumption stated in Theorem \ref{th:2scv-bdary}.  To get a stronger uniform bound, we apply the following scaling argument. Fix an $\epsilon$-periodic cell in $\Omega$, say the cell indexed by $\mathbf{k}$, the trace theorem implies:
\begin{align*}
\int_{\Gamma^{s}_{\epsilon}}|\widetilde{\mathbf{\hat{u}_{\epsilon,k}}^s}(\mathbf{x})|^2d\sigma_\epsilon(\mathbf{x})&=\epsilon^2  \int_{\Gamma^{{s}}}|\widetilde{\mathbf{\hat{u}_{\epsilon,k}^s}}(\epsilon\,\mathbf{y})|^2 d\sigma(\mathbf{y})
\\
&\le \epsilon^2 C(Y_s) \left(\int_{Y_s} |\mathbf{\hat{u}_{\epsilon,k}^s}(\epsilon\,\mathbf{y})|^2+|\nabla_{\mathbf{y}}\mathbf{\hat{u}_{\epsilon,k}^s}(\epsilon\,\mathbf{y})|^2\right) d\mathbf{y}\\
&=C(Y_s)\epsilon^2 \left( \int_{\epsilon Y_s}\epsilon^{-3}  |\mathbf{\hat{u}_{\epsilon,k}^s(\mathbf{x})}|^2+\epsilon^{-1}|\nabla\mathbf{\hat{u}_{\epsilon,k}^s(\mathbf{x})}|^2 d\mathbf{x}\right).
\end{align*}
Hence, by summing over all $\mathbf{k}$, we arrive at the bound needed for Theorem \ref{th:2scv-bdary}:
\[
\epsilon \int_{\Gamma^{s}_{\epsilon}}|\widetilde{\mathbf{\hat{u}_\epsilon}^s}|^2d\sigma_\epsilon(\mathbf{x}) \le C(Y_s)\left(\norm{\mathbf{\hat{u}_\epsilon}^s}^2_{L^2(\Omega_\epsilon^s)}+\epsilon^2\norm{\nabla\mathbf{\hat{u}_\epsilon}^s}^2_{L^2(\Omega_\epsilon^s)}\right)\leq C.
\]

By Theorem~\ref{th:2scv-bdary}, there exists $\mathbf{v}\in L^2(\Omega,L^2(\Gamma))$ such that:  $$\epsilon\int_{\Gamma^{s}_\epsilon}\widetilde{\mathbf{\hat{u}_\epsilon^s}}(\mathbf{x})\phi\left(\mathbf{x},\frac{\mathbf{x}}{\epsilon}\right)d\sigma_\epsilon(\mathbf{x})\rightarrow\int_{\Omega}\int_{{\Gamma}^{s}}\mathbf{v}(\mathbf{x},\mathbf{y})\phi(\mathbf{x},\mathbf{y})d\mathbf{x}d\sigma(\mathbf{y}), $$ for all $\phi(\mathbf{x},\mathbf{y})\in C[\overline{\Omega},C_p(Y)]$.  Following the proof of Proposition~2.6 in \cite{Allaire95}, for any vector-valued smooth test function $\boldsymbol{\varphi}(\mathbf{x},\mathbf{y})$, we have: 
\begin{align*}
	\epsilon\int_{\Omega_\epsilon^s}\nabla\mathbf{\hat{u}_\epsilon^s}(\mathbf{x})\boldsymbol{\varphi}\left(\mathbf{x},\frac{\mathbf{x}}{\epsilon}\right)d\mathbf{x}&=-\epsilon\int_{\Omega_\epsilon^s}\mathbf{\hat{u}_\epsilon^s}(\mathbf{x}){\rm div}_\mathbf{x}\boldsymbol{\varphi}\left(\mathbf{x},\frac{\mathbf{x}}{\epsilon}\right)d\mathbf{x}\\
	&-\int_{\Omega_\epsilon^s}\mathbf{\hat{u}_\epsilon^s}(\mathbf{x}){\rm div}_\mathbf{y}\boldsymbol{\varphi}\left(\mathbf{x},\frac{\mathbf{x}}{\epsilon}\right)d\mathbf{x}\\
	&+\epsilon\int_{\Gamma^{s}_\epsilon}\widetilde{\mathbf{\hat{u}_\epsilon^s}}(\mathbf{x}) (\boldsymbol{\varphi}\left(\mathbf{x},\frac{\mathbf{x}}{\epsilon}\right)\cdot \mathbf{n}_s) d\sigma_\epsilon(\mathbf{x}).
\end{align*}

Passing to the two-scale limit in each term, we obtain:
\begin{equation*}
	0=-\int_{\Omega}\int_{Y_s }\mathbf{u}(\mathbf{x}){\rm div}_\mathbf{y}\boldsymbol{\varphi}(\mathbf{x},\mathbf{y})d\mathbf{x}d\mathbf{y}+\int_{\Omega}\int_{\Gamma^{s}}\mathbf{v}(\mathbf{x},\mathbf{y}) (\boldsymbol{\varphi}\left(\mathbf{x},\mathbf{y}\right)\cdot \mathbf{n}_s )d\sigma(\mathbf{y}) d\mathbf{x}.
\end{equation*}
Therefore, $\displaystyle\int_{\Omega}\int_{\Gamma}(\mathbf{v}(\mathbf{x},\mathbf{y})-\mathbf{u}(\mathbf{x}))\varphi\left(\mathbf{x},\mathbf{y}\right)\cdot \mathbf{n}_s d\sigma(\mathbf{y}) d\mathbf{x}=0$, which implies that, for $\mathbf{y}\in \Gamma$, $\mathbf{v}(\mathbf{x},\mathbf{y})=\mathbf{u}(\mathbf{x})$ for all $\mathbf{x}\in\Omega$ and hence the following two-scale convergence result holds: 
\begin{equation}
	\lambda\,\epsilon\alpha\int_{\Gamma^{s}_\epsilon}\widetilde{\mathbf{\hat{u}_\epsilon}^s}\cdot \overline{\llbracket \mathbf{w^\epsilon}\phi\rrbracket}_s^f\,d\sigma_\epsilon(\mathbf{x})\rightarrow \lambda\,\alpha\int_{\Omega}\int_{\Gamma^{s}}\mathbf{u}(\mathbf{x})\overline{\phi(\mathbf{x})(\mathbf{w}^f(\mathbf{y})-\mathbf{w}^s(\mathbf{y}))}d\sigma(\mathbf{y}) d\mathbf{x}. \label{Gamma_solid}
\end{equation}

For the fluid part, we know that $\norm{\mathbf{\hat{u}_\epsilon}^f}_{L^2(\Omega_\epsilon^f)}$ and  $\norm{\epsilon\nabla\mathbf{\hat{u}_\epsilon}^f}_{L^2(\Omega_\epsilon^f)}$  are uniformly bounded. We use $\widetilde {\mathbf{\hat{u}_\epsilon}^f }$ (resp. $\epsilon\nabla\widetilde {\mathbf{\hat{u}_\epsilon}^f }$) to denote the extension by zero of $\mathbf{\hat{u}_\epsilon}^f$ (resp. $\epsilon\nabla\mathbf{\hat{u}_\epsilon}^f$), which is the restriction of $\mathbf{\hat{u}_\epsilon}$ (resp. $\epsilon\nabla\mathbf{\hat{u}_\epsilon}$) to ${\Omega_\epsilon^f}$ and apply similar arguments as above.  By Proposition~1.14(ii) of \cite{Allaire1992}, there exist $\boldsymbol{\zeta}\in L^2(\Omega;(H_p^1(Y_f))^3)$ and $\boldsymbol{\xi}\in \mathbf{L^2(\Omega;(H_p^1(Y_f))^9)}$ such that $\widetilde{\mathbf{\hat{u}_\epsilon}^f}$ two-scale converges to $\boldsymbol{\zeta}$ and $\epsilon\nabla\widetilde{\mathbf{\hat{u}_\epsilon}^f}$ two-scale converges to $\boldsymbol{\xi}$, with $\boldsymbol{\zeta}(\mathbf{x},\mathbf{y})=0$ for $\mathbf{y}\in Y_s$ and $\boldsymbol{\xi}(\mathbf{x},\mathbf{y})=0$ for $\mathbf{y}\in Y_s$.   For any test function $\varphi(\mathbf{x},\mathbf{y})\in \mathscr{D}(\Omega;C_p^\infty(Y))$ and $\boldsymbol{\psi}(\mathbf{x},\mathbf{y})\in \boldsymbol{\mathscr{D}}(\Omega;C_p^\infty(Y))$, with $\varphi(\mathbf{x},\mathbf{y})=0$ for $\mathbf{y}\in Y_s$ and $\boldsymbol{\psi}(\mathbf{x},\mathbf{y})=0$ for $\mathbf{y}\in Y_s$, we have: 
\begin{align*}
	\lim_{\epsilon\to0}\int_{\Omega_\epsilon^f}(\mathbf{\hat{u}_\epsilon^f})(\mathbf{x})\varphi\left(\mathbf{x},\frac{\mathbf{x}}{\epsilon}\right)d\mathbf{x}&=\lim_{\epsilon\to0}\int_{\Omega}(\widetilde{\mathbf{\hat{u}_\epsilon^f}})(\mathbf{x})\varphi\left(\mathbf{x},\frac{\mathbf{x}}{\epsilon}\right)d\mathbf{x}\\
	&=\int_{\Omega}\int_{Y_f }\boldsymbol{\zeta}(\mathbf{x},\mathbf{y})\varphi(\mathbf{x},\mathbf{y})d\mathbf{y}d\mathbf{x},\\
	\lim_{\epsilon\to0}\epsilon\int_{\Omega_\epsilon^f}\nabla\mathbf{\hat{u}_\epsilon^f}(\mathbf{x})\cdot\boldsymbol{\psi}\left(\mathbf{x},\frac{\mathbf{x}}{\epsilon}\right)d\mathbf{x}&=\int_{\Omega}\int_{Y_f }\boldsymbol{\xi}(\mathbf{x},\mathbf{y})\cdot\boldsymbol{\psi}(\mathbf{x},\mathbf{y})d\mathbf{y}d\mathbf{x}.
\end{align*}

Observe that:
\begin{align*}
	\epsilon\int_{\Omega_\epsilon^f}\nabla\mathbf{\hat{u}_\epsilon^f}(\mathbf{x})\cdot\boldsymbol{\psi}\left(\mathbf{x},\frac{\mathbf{x}}{\epsilon}\right)d\mathbf{x}&=-\epsilon\int_{\Omega_\epsilon^f}\mathbf{\hat{u}_\epsilon^f}(\mathbf{x}){\rm div}_\mathbf{x}\boldsymbol{\psi}\left(\mathbf{x},\frac{\mathbf{x}}{\epsilon}\right)d\mathbf{x}\\
	&\quad-\int_{\Omega_\epsilon^f}\mathbf{\hat{u}_\epsilon^f}(\mathbf{x}){\rm div}_\mathbf{y}\boldsymbol{\psi}\left(\mathbf{x},\frac{\mathbf{x}}{\epsilon}\right)d\mathbf{x}.
\end{align*}
Passing to the two-scale limit in each term above, we obtain the relation between $\boldsymbol{\zeta}$ and $\boldsymbol{\xi}$:
\begin{equation*}
	\int_{\Omega}\int_{Y_f }\boldsymbol{\xi}(\mathbf{x},\mathbf{y})\cdot\psi(\mathbf{x},\mathbf{y})d\mathbf{y}d\mathbf{x}=-\int_{\Omega}\int_{Y_f }\boldsymbol{\zeta}(\mathbf{x},\mathbf{y}){\rm div}_\mathbf{y}\psi(\mathbf{x},\mathbf{y})d\mathbf{y}d\mathbf{x};
\end{equation*}
from which we have $\boldsymbol{\xi}=\nabla_y\boldsymbol{\zeta}$, for $j=1,2,3$. Therefore, $\widetilde{\mathbf{\hat{u}}_\epsilon^f}$ two-scale converges to $\boldsymbol{\zeta}(\mathbf{x},\mathbf{y})\chi(Y_f)$ and $\epsilon\nabla\widetilde{\mathbf{\hat{u}}_\epsilon^f}$ two-scale converges to $\boldsymbol{\xi}(\mathbf{x},\mathbf{y})\chi(Y_f)=\nabla_y\boldsymbol{\zeta}(\mathbf{x},\mathbf{y})\chi(Y_f)$.   

Applying the scaled trace inequality in $\Omega_\epsilon^f$ by regarding $\Gamma^{f}_\epsilon=\partial\Omega
_\epsilon^f$, leads to:  $$\epsilon\int_{\Gamma^{f}_\epsilon}|\widetilde{\mathbf{\hat{u}_\epsilon}^f}|^2d\sigma_\epsilon(\mathbf{x})\leq C\norm{\mathbf{\hat{u}_\epsilon}^f}^2_{L^2(\Omega_\epsilon^f)}+\epsilon^2\norm{\nabla\mathbf{\hat{u}_\epsilon}^f}^2_{L^2(\Omega_\epsilon^f)}\leq C.$$ By Theorem~\ref{th:2scv-bdary}, we have that there exists $\mathbf{h}\in \mathbf{L^2}(\Omega,\mathbf{L^2}(\Gamma))$ such that: $$\epsilon\int_{\Gamma^{f}_\epsilon}\widetilde{\mathbf{\hat{u}_\epsilon^f}}(\mathbf{x})\phi\left(\mathbf{x},\frac{\mathbf{x}}{\epsilon}\right)d\sigma_\epsilon(\mathbf{x})\rightarrow\int_{\Omega}\int_{\Gamma^{f}}\mathbf{h}(\mathbf{x},\mathbf{y})\phi(\mathbf{x},\mathbf{y})d\sigma(\mathbf{y})d\mathbf{x},$$ for all $\phi(\mathbf{x},\mathbf{y})\in C[\overline{\Omega},C_p(Y)]$.  Note that, for any vector-valued smooth test function $\boldsymbol{\varphi}(\mathbf{x},\mathbf{y})$, we have: 
\begin{align*}
	\epsilon\int_{\Omega_\epsilon^f}\nabla\mathbf{\hat{u}_\epsilon^f}(\mathbf{x})\cdot\boldsymbol{\varphi}\left(\mathbf{x},\frac{\mathbf{x}}{\epsilon}\right)d\mathbf{x}&=-\epsilon\int_{\Omega_\epsilon^f}\mathbf{\hat{u}_\epsilon^f}(\mathbf{x}){\rm div}_\mathbf{x} \boldsymbol{\varphi}\left(\mathbf{x},\frac{\mathbf{x}}{\epsilon}\right)d\mathbf{x}\\
	&\quad-\int_{\Omega_\epsilon^f}\mathbf{\hat{u}_\epsilon^f}(\mathbf{x}){\rm div}_\mathbf{y}\boldsymbol{\varphi}\left(\mathbf{x},\frac{\mathbf{x}}{\epsilon}\right)d\mathbf{x}\\
	&\quad+\epsilon\int_{\Gamma^{f}_\epsilon}\widetilde{\mathbf{\hat{u}_\epsilon^f}}(\mathbf{x})\boldsymbol{\varphi}\left(\mathbf{x},\frac{\mathbf{x}}{\epsilon}\right)\cdot \mathbf{n} \,d\sigma_\epsilon(\mathbf{x}).
\end{align*}
Passing to the two-scale limit in each term, we obtain:
\begin{align*}
	\int_{\Omega}\int_{Y_f }\boldsymbol{\xi}(\mathbf{x},\mathbf{y})\cdot \boldsymbol{\varphi}(\mathbf{x},\mathbf{y})d\mathbf{y}d\mathbf{x}&=\int_{\Omega}\int_{Y_f }\nabla_\mathbf{y}\boldsymbol{\zeta}(\mathbf{x},\mathbf{y})\cdot\boldsymbol{\varphi}(\mathbf{x},\mathbf{y})d\mathbf{y}d\mathbf{x}\\
	&=-\int_{\Omega}\int_{Y_f }\boldsymbol{\zeta}(\mathbf{x},\mathbf{y}){\rm div}_\mathbf{y}\boldsymbol{\varphi}(\mathbf{x},\mathbf{y})d\mathbf{y}d\mathbf{x}\\
	&\quad+\int_{\Omega}\int_{\Gamma^{f}}\mathbf{h}(\mathbf{x},\mathbf{y})\boldsymbol{\varphi}\left(\mathbf{x},\mathbf{y}\right)\cdot \mathbf{n} \,d\sigma(\mathbf{y}) d\mathbf{x}.
\end{align*}
An application of integration by parts to the second integral above leads to: $$\int_{\Omega}\int_{\Gamma^{f}}(\mathbf{h}(\mathbf{x},\mathbf{y})-\boldsymbol{\zeta}(\mathbf{x},\mathbf{y}))\boldsymbol{\varphi}\left(\mathbf{x},\mathbf{y}\right)\cdot \mathbf{n} \,d\sigma(\mathbf{y}) d\mathbf{x}=0,$$ which implies that $\boldsymbol{\zeta}(\mathbf{x},\mathbf{y})\big|_{\Gamma^{f}}=\mathbf{h}(\mathbf{x},\mathbf{y})$.  Recall that $\boldsymbol{\zeta}(\mathbf{x},\mathbf{y})=\mathbf{u}(\mathbf{x})+\mathbf{u_r}(\mathbf{x},\mathbf{y})$ for $\mathbf{y}\in Y_f$.

Then, we obtain the following two-scale convergence result:
\begin{align}
	&\lambda\,\epsilon\alpha\int_{\Gamma^{f}_\epsilon}\widetilde{\mathbf{\hat{u}_\epsilon}^f}\cdot \overline{\llbracket \mathbf{w^\epsilon}\phi\rrbracket}_s^f\,d\sigma_\epsilon(\mathbf{x})\notag\\
	&\rightarrow \lambda\,\alpha\int_{\Omega}\int_{\Gamma^{f}}(\mathbf{u}(\mathbf{x})+\mathbf{u_r}(\mathbf{x},\mathbf{y}))\overline{\phi(\mathbf{x})(\mathbf{w}^f(\mathbf{y})-\mathbf{w}^s(\mathbf{y}))}d\sigma(\mathbf{y}) d\mathbf{x}. \label{Gamma_fluid}
\end{align}
The lemma is then proved by subtracting \eqref{Gamma_solid} from \eqref{Gamma_fluid}.
\end{proof}

{The following Lemma will also be needed in the proof of Theorem \ref{sec:probur}}.
\begin{lemma}
\label{lem:e2pu}
	Let $\epsilon$ be the subsequence involved in Lemma~\ref{lem:2scalecvuep}.  Then, as $\epsilon\downarrow0$, the following holds:
\begin{equation*}
	\epsilon^2 b^\epsilon(\hat{\mathbf{u}}_\epsilon,\mathbf{w}^\epsilon)=\epsilon^2\int_{\Omega}b^\epsilon_{ijkl}\frac{\partial \hat{u}_\epsilon^k}{\partial x_l}\frac{\partial{w^i}^\epsilon}{\partial x_j}\phi\,d\mathbf{x}\rightarrow\mu_{ijkl}\int_{\Omega\times Y_f}\frac{\partial w_o^k}{\partial y_l}(\mathbf{x},\mathbf{y})\frac{\partial w^i}{\partial y_j}(\mathbf{y})\phi(\mathbf{x})\,d\mathbf{x}\,d\mathbf{y},
\end{equation*}
for all $\mathbf{w}=\left\{w^i\right\}\in\mathbf{H^1_p}(Y_f)$, for all $\phi\in\mathscr{K}(\overline{\Omega})$, where $\mu_{ijkl}=\mu(\delta_{ik}\delta_{jl}+\delta_{jk}\delta_{il})$.
\end{lemma}
\begin{proof}
Notice that $\displaystyle\frac{\partial}{\partial x_j}({w^i}^\epsilon(x))=\frac{1}{\epsilon}\frac{\partial w^i}{\partial y_j}\big|_{y=\frac{x}{\epsilon}}$.  Choosing $\psi=\displaystyle b_{ijkl}\,\frac{\partial w^i}{\partial y_j}$ in (\ref{2sclduep}), with $\mathbf{w}\in\mathbf{H^1_p}(Y_f)$, we obtain the result, due to (\ref{divwo}) and $\psi^\epsilon=\displaystyle \epsilon \,b^\epsilon_{ijkl}\,\frac{\partial {w^i}^\epsilon}{\partial x_j}$.
\end{proof}
Now, we have all the  ingredients needed for proving Theorem \ref{sec:probur}.
\begin{proof}[Proof of Theorem \ref{sec:probur}]
For $\mathbf{u_r}$, we first note that $\mathbf{u_r}\in \mathbf{L^2}(\Omega;W)$; see Lemma~\ref{lem:due2scsol}. Now, we use the test function $\mathbf{w^\epsilon}\phi$, $\mathbf{w}\in W$ and $\phi\in\mathscr{D}(\Omega)$ in (\ref{LVF}), to obtain:
\begin{align}
\label{equrproblemeps}
	\int_{\Omega_\epsilon^f}\hat{\mathbf{f}}\cdot\overline{\mathbf{w}^\epsilon\phi}\,d\mathbf{x}&=\lambda^2\int_{\Omega_\epsilon^f}\rho^f  \hat{\mathbf{u}}_\epsilon\cdot \overline{{\mathbf{w}^\epsilon}\phi} \,d\mathbf{x}+c^{\epsilon}(\mathbf{\hat{u}_\epsilon},\mathbf{w^\epsilon}\phi)\notag\\
&
\quad+\lambda\epsilon^2b^\epsilon\left(\mathbf{\hat{u}_\epsilon},\mathbf{w^\epsilon}\phi\right)+\lambda\,\epsilon\int_{\Gamma_\epsilon}\alpha\llbracket\mathbf{\hat{u}_\epsilon}\rrbracket_s^f\cdot \overline{\llbracket\mathbf{w^\epsilon}\phi\rrbracket}_s^f\,d\sigma_\epsilon(\mathbf{x}).
\end{align}
Then, by letting $\epsilon\downarrow0$ in (\ref{equrproblemeps}), and using Lemmas~\ref{lem:2scalecvuep}, \ref{lem:e2pu}, and \ref{lem:pe2sv}, we obtain:
\begin{align}
	&\int_{\Omega\times Y_f}\hat{\mathbf{f}}(\mathbf{x})\cdot\overline{\mathbf{w}(\mathbf{y})\phi(\mathbf{x})}\,d\mathbf{x}\,d\mathbf{y}\label{lab_1}\\
	&=\lambda^2\int_{\Omega\times Y_f}\rho^f(\mathbf{u}(\mathbf{x})+\mathbf{u_r}(\mathbf{x},\mathbf{y}))\cdot \overline{\mathbf{w}(\mathbf{y})\phi(\mathbf{x})}\,d\mathbf{x}\notag \\
	&\quad+\lim_{\epsilon\to0}c^{\epsilon}(\mathbf{\hat{u}_\epsilon},\mathbf{w^\epsilon}\phi)+\lambda\epsilon^2b^\epsilon\left(\mathbf{\hat{u}_\epsilon},\mathbf{w^\epsilon}\phi\right)\notag\\
	&\quad+\lim_{\epsilon\to 0} \left(\lambda\,\epsilon\,\alpha\int_{\Gamma_\epsilon^f}\mathbf{\hat{u}_\epsilon}\cdot \overline{\llbracket\mathbf{w^\epsilon}\phi\rrbracket}_s^f\,d\sigma_\epsilon(\mathbf{x})-\lambda\,\epsilon\alpha\int_{\Gamma_\epsilon^s}\mathbf{\hat{u}_\epsilon}\cdot \overline{\llbracket\mathbf{w^\epsilon}\phi\rrbracket}_s^f\,d\sigma_\epsilon(\mathbf{x}) \right).\notag
\end{align}
With a straightforward calculation, applying Lemma~\ref{lem:2scalecvuep} and Lemma~\ref{lem:pe2sv}, and integrating by parts, the following limits can be concluded
\begin{align}
\lim_{\epsilon\to 0} \lambda\epsilon^2b^\epsilon\left(\mathbf{\hat{u}_\epsilon},\mathbf{w^\epsilon}\phi\right)&=2\mu\lambda \int_\Omega\int_{Y_f} \overline{\phi(\mathbf{x})}\overline{e_{ij}(\mathbf{w})}e_{ij}(\mathbf{u}(\mathbf{x})+\mathbf{u_r}(\mathbf{x},\mathbf{y}))\, d\mathbf{y}\,d\mathbf{x}\label{lab_2} \\
&= 2\mu\lambda \int_\Omega\int_{Y_f} \overline{\phi(\mathbf{x})}\overline{\frac{\partial w^i}{\partial y_j}} \frac{\partial u_r^i}{\partial y_j}\, d\mathbf{y}\,d\mathbf{x}.\notag \\
\lim_{\epsilon\to 0} c^{\epsilon}(\mathbf{\hat{u}_\epsilon},\mathbf{w^\epsilon}\phi)&= \int_\Omega \int_{Y_f} \nabla p_0(\mathbf{x})  \overline{\mathbf{w}(\mathbf{y}) \phi(\mathbf{x})}\, d\mathbf{y}\,d\mathbf{x}. \label{lab_3}
\end{align}
Taking into account Lemma \ref{l:twoscaleconvinterf} and the two equations above, \eqref{lab_1} leads to the problem for $\mathbf{u_r}$ in \eqref{locprob-ur}.
\end{proof}
In preparation for deriving the homogenized equations in the next section, we calculate $\mathbf{u_1}$ in terms of $\mathbf{u}$ and $p_0$.  To do this, we seek a solution of the form: 
\begin{equation}
\label{u1-formula}
\mathbf{u_1}(\mathbf{x},\mathbf{y})=-\frac{\partial u^k}{\partial x_l}(\mathbf{x})\boldsymbol{\chi_k^l}(\mathbf{y})-p_0(\mathbf{x})\boldsymbol{\chi}(\mathbf{y}).
\end{equation}
with $\boldsymbol{\chi}$, $\boldsymbol{\chi_i^j}\in\mathbf{H^1_p}(Y_s)/\mathbb{C}^3$, $1\leq i,j\leq3$, real-valued vector functions, independent of $\mathbf{x}$.  It can be verified with a straightforward calculation that the vectors $\boldsymbol{\chi}$ and $\boldsymbol{\chi_i^j}$ satisfy the following equations:
\begin{align}
	q\left(\boldsymbol{\chi},\mathbf{w}\right)&=\int_{Y_s}\overline{{\rm div}_\mathbf{y}\mathbf{w}}\,d\mathbf{y},\hspace{3mm}\forall \mathbf{w}\in\mathbf{H^1_p}(Y_s)/\mathbb{C}^3,\label{a-loc-chi}\\
	q\big(\boldsymbol{\chi_i^j},\mathbf{w}\big)&=\int_{Y_s}a_{ijkl}\overline{\frac{\partial w^k}{\partial y_l}}\,d\mathbf{y},\hspace{3mm}\forall \mathbf{w}\in\mathbf{H^1_p}(Y_s)/\mathbb{C}^3,\label{a-loc-chi-ij}
\end{align}
respectively, which are uniquely defined by (\ref{a-coerc}) and are independent of $\lambda$.

We set:
\begin{equation*}
	\beta_{ij}=-\int_{Y_s}{\rm div}_\mathbf{y}\boldsymbol{\chi_i^j}\,d\mathbf{y},\hspace{3mm}\beta=\int_{Y_s}{\rm div}_\mathbf{y}\boldsymbol{\chi}\,d\mathbf{y}.
\end{equation*}
Note that $\beta=q(\boldsymbol{\chi},\boldsymbol{\chi})\ge0$ (see (\ref{a-loc-chi})).

Equation (\ref{u1-formula}) allows us to write $p_0(\mathbf{x})$ in terms of $\mathbf{u}(\mathbf{x})$ and $\langle\mathbf{u_r}\rangle(\mathbf{x})$ as follows.  By substitution of (\ref{u1-formula}) into (\ref{s-f-rel}), we get:
\begin{equation}
	\label{p0uuep}
	\delta^{-1}p_0=\beta_{kl}\frac{\partial u^k}{\partial x_l}-\Pi\,{\rm div}\,\mathbf{u}-{\rm div}\,\langle\mathbf{u_r}\rangle,
\end{equation}
where $\delta$ and $\Pi$ are given by:
\begin{equation*}
	\delta=\left(\frac{\Pi}{\gamma}+\beta\right)^{-1}>0,\hspace{4mm}\Pi=\frac{\left|Y_f\right|}{\left|Y\right|}=|Y_f|>0.
\end{equation*}

\section{The homogenized problem.} 
\label{HomogProbConvTh}

{In this section, we derive the governing equations for $p_0(\mathbf{x})$ and $\mathbf{w_0}$, the two-scale limit of $\mathbf{\hat{u}_\epsilon}$. As will be seen in the theorem below, this homogenized problem is posed in a six dimensional space for $\mathbf{u}$ and $\mathbf{u_r}$.}

\begin{theorem}
\label{Macroscopic equations}
For every $\lambda$ such that $\lambda>\lambda_0>r>1$, the homogenized problem for $\mathbf{w_0}=\mathbf{u}+\mathbf{u_r}$ is the solution to the following uniquely solvable equation:
\begin{align}
\label{maceq-conn-fin}
	&\frac{1}{\lambda}\int_{\Omega}\hat{f}^i\overline{(w^i+\langle\mathbf{w_r}\rangle^i)}\,d\mathbf{x}\\
	&\quad=\lambda\rho^f\int_{\Omega\times Y_f}(u^i+u^i_r)\overline{(w^i+w^i_r)}\,d\mathbf{x}\,d\mathbf{y}+\lambda\left(1-\Pi\right)\rho^s\int_{\Omega}u^i\overline{w^i}\,d\mathbf{x}\notag\\
	&\qquad+\frac{1}{\lambda}\int_{\Omega}q_{ijlk}\frac{\partial u^k}{\partial x_l}\overline{\frac{\partial w^i}{\partial x_j}}\,d\mathbf{x}+2\mu\int_{\Omega\times Y_f}\frac{\partial u_r^i}{\partial y_j}\overline{\frac{\partial w_r^i}{\partial y_j}}\,d\mathbf{x}\,d\mathbf{y}\notag\\
	&\qquad+\alpha\int_{\Omega}\int_{\Gamma}\mathbf{u_r}(\mathbf{x},\mathbf{y})\cdot\overline{\mathbf{w_r}(\mathbf{x},\mathbf{y})}\,d\sigma(\mathbf{y})\,d\mathbf{x}\notag\\
	&\qquad+\frac{\delta}{\lambda}\int_{\Omega}\left(\beta_{kl}\frac{\partial u^k}{\partial x_l}-\Pi{\rm div}\mathbf{u}-{\rm div}\langle\mathbf{u_r}\rangle\right)\left(\beta_{ij}\overline{\frac{\partial w^i}{\partial x_j}}-\Pi\overline{{\rm div}\mathbf{w}}-\overline{{\rm div}\langle\mathbf{w_r}\rangle}\right)\,d\mathbf{x},\notag\\
	&=:F(\mathbf{u}+\mathbf{u_r},\mathbf{w}+\mathbf{w_r})\notag
	\end{align}
for all $\mathbf{w}\in\mathbf{H_0^1}(\Omega)$ and all $\mathbf{w_r}\in\mathbf{H}(\Omega;W)$, which is defined as 
$$\mathbf{H}(\Omega;W)=\left\{\mathbf{w}\,:\,\mathbf{w}\in\mathbf{L^2}(\Omega,W),\langle\mathbf{w}\rangle\in E_0(\Omega_\epsilon^s\cup\Omega_\epsilon^f)\right\},$$ 
and is a Hilbert space with the norm:  
$$\norm{\mathbf{w}}_{\mathbf{H}}=\left(\norm{\mathbf{w}}^2_{\mathbf{L^2}(\Omega;W)}+\norm{{\rm div}\langle\mathbf{w}\rangle}^2_{\mathbf{L^2}(\Omega)}\right)^{1/2}.$$
\end{theorem}
\begin{proof}
We take $\mathbf{w}\in\mathscr{D}(\Omega)$ in (\ref{LVF}), and concentrate on passing to the limit as $\epsilon\downarrow0$, using Lemmas~\ref{2scaleth},~\ref{lem:2scalecvuep} and \ref{lem:due2scsol}.  Observe that, if $\mathbf{w}\in\mathscr{D}(\Omega)$, the interface term drops automatically because $\mathbf{w}_f=\mathbf{w}_s$ on $\Gamma_\epsilon$, obtaining:
\begin{align}
	\label{macprob-ur}
\int_{\Omega}\hat{f}^i\overline{w^i}\,d\mathbf{x}	&=\lambda^2\int_{\Omega}\left(\langle\rho\rangle u^i+\rho^f\langle u_r\rangle^i\right)\overline{w^i}\,d\mathbf{x}-|Y_f|\int_{\Omega}p_0\,\overline{{\rm div}\,\mathbf{w}}\,d\mathbf{x}\notag\\
	&\quad+\int_{\Omega\times Y_s}a_{ijkl}\left(\frac{\partial u^k}{\partial x_l}+\frac{\partial u_1^k}{\partial y_l}\right)\overline{\frac{\partial w^i}{\partial x_j}}\,d\mathbf{x}\,d\mathbf{y},
\end{align} 
for all $\mathbf{w}\in\mathscr{D}(\Omega)$, where we recall the notation that is already defined in \eqref{uowo}: 
\begin{equation*}
\langle\mathbf{v}\rangle(\mathbf{x})=\int_{Y}\mathbf{v}(\mathbf{x},\mathbf{y})\,d\mathbf{y}\hspace{3mm}\text{ for }\mathbf{v}\in\mathbf{L^2}(\Omega,\mathbf{L_p^2}).
\end{equation*}
In order to replace the $\mathbf{u}_1$ term with the zero-order terms $\mathbf{u}_0$ and $p_0$, we use the solutions of the cell problems for $\mathbf{u}_1$  \eqref{u1-formula}-\eqref{a-loc-chi-ij} to define the following auxiliary variables. 
For $1\leq i,j\leq3$, let $\mathbf{p_i^j}:=y_j\delta_{ik}$,  $k=1,2,3,$ and introduce:
\begin{align*}
	q_{ijkl}:=q\left(\boldsymbol{\chi_i^j}-\mathbf{p_i^j},\boldsymbol{\chi_k^l}-\mathbf{p_k^l}\right),
\end{align*}
where $q(\cdot,\cdot)$ (respectively, $\boldsymbol{\chi_i^j}$) is defined in (\ref{a}) (respectively, (\ref{a-loc-chi-ij})).   Observe that the coefficients $q_{ijkl}$ are real and they satisfy:
\begin{align}
&q_{ijkl}=q_{jikl}=q_{ijlk}=q_{klij},\notag\\
&q_{ijkl}\xi_{kl}\xi_{ij}\geq c\,\xi_{ij}\xi_{ij}\hspace{3mm}(c>0)\hspace{3mm}\xi_{ij}=\xi_{ji}\hspace{2mm}(1\leq i,j\leq3).\label{ellipq}
\end{align}

A calculation shows that: $$\int_{Y_s}a_{ijlk}\left(\frac{\partial u^k}{\partial x_l}+\frac{\partial u_1^k}{\partial y_l}\right)\,d\mathbf{y}=q_{ijlk}\frac{\partial u^k}{\partial x_l}+\beta_{ij}p_0,\,\, i,j,k,l=1,2,3.
$$
We substitute the above equation into (\ref{macprob-ur}) and use the fact that $\mathscr{D}(\Omega)$ is dense in $\mathbf{H_0^1}(\Omega)$, to obtain the macroscopic equation:
\begin{align}
\label{maceq}
	\int_{\Omega}\hat{f}^i\overline{w^i}\,d\mathbf{x}&=\lambda^2\int_{\Omega}\left(\langle\rho\rangle u^i+\rho^f\langle u_r\rangle^i\right)\overline{w^i}\,d\mathbf{x}+\int_{\Omega}q_{ijlk}\frac{\partial u^k}{\partial x_l}\overline{\frac{\partial w^i}{\partial x_j}}\,d\mathbf{x}\notag\\
	&\quad+\delta\int_{\Omega}\left(\beta_{kl}\frac{\partial u^k}{\partial x_l}-\Pi{\rm div}\mathbf{u}-{\rm div}\langle\mathbf{u_r}\rangle\right)\left(\beta_{ij}\overline{\frac{\partial w^i}{\partial x_j}}-\Pi\overline{{\rm div}\mathbf{w}}\right)\,d\mathbf{x},
\end{align}
for all $\mathbf{w}\in\mathbf{H_0^1}(\Omega)$. 

{To close the system, we substitute} (\ref{p0uuep}) into (\ref{locprob-ur}) and test it with  functions of the form $\mathbf{w}=\mathbf{w_r}(\mathbf{x},\cdot)$ for fixed $\mathbf{x}$, $\mathbf{w_r}\in\mathscr{D}(\Omega;W)$ (W defined in (\ref{spaceW})), followed by integrating over $\Omega$ to obtain:
\begin{align}
	\label{macprob-connec}
	\int_{\Omega}\hat{f}^i\overline{\langle\mathbf{w_r}\rangle^i}\,d\mathbf{x}&=\lambda^2\rho^f\int_{\Omega\times Y_f}\left(u^i_r+u^i\right)\overline{w^i_r}\,d\mathbf{x}\,d\mathbf{y}+2\lambda\mu\int_{\Omega\times Y_f}\frac{\partial u_r^i}{\partial y_j}\overline{\frac{\partial w_r^i}{\partial y_j}}\,d\mathbf{x}\,d\mathbf{y}\notag\\
	&\quad+\alpha\lambda\int_\Omega\int_\Gamma\mathbf{u_r}(\mathbf{x},\mathbf{y})\overline{\mathbf{w_r}(\mathbf{x},\mathbf{y})}d\sigma(\mathbf{y})d\mathbf{x}\notag\\
	&\quad+\delta\int_{\Omega}\left(\beta_{kl}\frac{\partial u^k}{\partial x_l}-\Pi{\rm div}\mathbf{u}-{\rm div}\langle\mathbf{u_r}\rangle\right)\overline{\left(-{\rm div}\langle\mathbf{w_r}\rangle\right)}\,d\mathbf{x},
\end{align}
for all $\mathbf{w_r}\in\mathscr{D}(\Omega;W)$.  The space $\mathbf{H}$ is chosen because $\mathbf{u_r}\in L^2(\Omega,W)$ and $\langle \mathbf{u_r} \rangle=\mathbf{\hat{u}_0}-\mathbf{u}\in E_0(\Omega_\epsilon^s\cup\Omega_\epsilon^f)$, i.e. $\mathbf{u_r}\in \mathbf{H}(\Omega;W)$. Since 
$\mathscr{D}(\Omega;W)$ is dense in the space $\mathbf{H}(\Omega;W)$, we can replace $\mathscr{D}(\Omega;W)$ by the space $\mathbf{H}(\Omega;W)$ in (\ref{macprob-connec}).

Note that the first integral on the right-hand side of  (\ref{maceq}) can be written as: $$\left(1-\Pi\right)\rho^s\int_{\Omega}u^i\overline{w^i}\,d\mathbf{x}+\rho^f\int_{\Omega\times Y_f}(u^i+u^i_r)\overline{w^i}\,d\mathbf{x}\,d\mathbf{y}.$$ Combining (\ref{maceq}) with (\ref{macprob-connec}) and divide both side by $\lambda$ lead to problem \eqref{maceq-conn-fin}. It can be checked that $Re F(\mathbf{w+w_r}, \mathbf{w+w_r})$ is coercive and hence the existence and uniqueness of solution follow from the Lax-Milgram lemma.
\end{proof}

The time domain macroscopic equation can be obtained by applying inverse Laplace transform to the equation above.

\section{Conclusion}
\label{conclusion}
In this paper, we consider wave propagation in a poroelastic composite material. It generalizes the results obtained in \cite{Nguetseng1990} from no-slip condition on the solid-fluid interface to the case of a slip boundary condition given by the interface term (\ref{interfacecond}).  To handle this interface condition, various function spaces are defined in Section \ref{fun_spaces} to accommodate the discontinuity of $\mathbf{u}_\epsilon$ on the interface. 

The existence and uniqueness result presented in Section~\ref{EUProof} dealt with the interface term. 

Unlike \cite{Nguetseng1990}, this slip problem requires taking the two-scale convergence limit for a surface integral. The results from \cite{Allaire95} (presented in Section~\ref{2scaleboundsec}) generalize the definition of two-scale convergence to surfaces and are fundamental in the limiting process of the the interface term. We can use these results since we are able to obtain (\ref{unifbound}) and (\ref{unifboundgamma}). An important part of our analysis is to establish the relation between the two-scale limits of the functions and the two-scale limits of their traces. Another difference between our results and those in \cite{Nguetseng1990} is that we need to add the norm of the interface jump term to the $V$-norm so results like Lemma \ref{1-coercive} can hold. 

The interface term does not show up in the local problem for $\mathbf{u^1}$, see (\ref{locprou1-2}), and we obtained similar results to \cite{Nguetseng1990}.  However, the interface term is in the local problem for $\mathbf{u_r}$, which is obtained in Section~\ref{sec:probur}. Note that in \ref{equrproblemeps} the boundary term doesn't disappear, and the technical Lemma~\ref{l:twoscaleconvinterf} is necessary for dealing with this term and to finally obtain (\ref{locprob-ur}).  

Unlike the results in \cite{Nguetseng1990}, where the macroscopic equation in the case of inclusions has simpler form than the case of connected geometry, the macroscopic equations in the slip case are indifferent to whether the pore space is connected or not.  
 
The homogenized equations \eqref{maceq-conn-fin} are posed in six dimension space. Since $\mathbf{u_r}$ in \eqref{cell_ur} is linearly proportional to the force term: $$F^i(\mathbf{x}):=\left(\hat{f}^i(\mathbf{x})-\lambda^2\rho^fu^i(\mathbf{x})-\frac{\partial p_0}{\partial x_i}(\mathbf{x})\right),$$ we could have defined the auxiliary matrix-valued variable ${\boldsymbol{\theta}}$,  such that:
\begin{equation*}
u_r^i(\mathbf{x},\mathbf{y})=\theta_{ip}(\mathbf{y})F^p(\mathbf{x}).
\end{equation*} 
By substituting this expression into \eqref{locprob-ur}, the following equations for  $\boldsymbol{\theta}(\mathbf{y})$ can be easily obtained:
\begin{align}
\label{aux_ur}
&\lambda^2 \rho_f \int_{Y_f} \theta_{ip} \overline{w^i} d\mathbf{y}+2\lambda\mu \int_{Y_f}\frac{\partial \theta_{ip}}{\partial y_j}\overline{\frac{\partial w^i}{\partial y_j} } d\mathbf{y}+\lambda \alpha \int_\Gamma \theta_{ip} \overline{w^i} d\sigma(\mathbf{y})\\
\notag
&\quad=\int_{Y_f} \overline{\mathbf{w}}\cdot \mathbf{e}_pd\mathbf{y},\, p=1,2,3, \hspace{3mm} \forall\mathbf{w}\in W.   
\end{align}
This cell problem can be solved first and then the homogenized equation will be only for $\mathbf{u}(\mathbf{x})$ and hence a problem in three dimensions, instead of six. However, unlike the auxiliary variables introduced for $\mathbf{u_1}$, whose governing equations \eqref{a-loc-chi} and \eqref{a-loc-chi-ij}, $\boldsymbol{\theta}$ are independent of  $\lambda$, equation \eqref{aux_ur} depends on $\lambda$. This means that the corresponding three dimensional macroscopic equation  problem in the time-domain will contain memory terms with the inverse Laplace transform of $\boldsymbol{\theta}$ being the kernel function. Finally, we remark that as a result of the slip interface condition, the cell problem for $\mathbf{u_r}$ in \eqref{cell_ur} has the form of a generalized Darcy's law but with an additional term of $2\lambda\mu \frac{\partial e_{ij}(\mathbf{u_r})}{\partial y_j}$ . The consequence of this term on the permeability will be studied in the future work.
\section*{Acknowledgements}
The work of MYO was partially sponsored by NSF grants DMS-1413039 and DMS-1821857. The work of the SJB was partially supported by NSF grant DMS-2110036.  

\bibliographystyle{plain}

\appendix
\section{Useful lemmas used in the paper}
\label{useful_lemmas}
\numberwithin{equation}{section}
\numberwithin{theorem}{section}
\begin{lemma}
\label{laxmil}
Lax-Milgram~lemma (Theorem~5.1, page~18 of \cite{SanchezPalencia1980}).  If $a(u,v)$ is a sesquilinear form on $V$ such that: 
\begin{itemize}
	\item $a(\lambda u,\mu v)=\lambda\, \bar{\mu} \,a(u,v)$,
	\item $|a(u,v)|\leq M\norm{u}_V\norm{v}_V$,
\end{itemize}
and if there exists $C>0$ such that $|a(u,u)|\geq C\norm{u}_V^2$ for all $u\in V$, then, for every $f\in V'$ (the dual space of $V$), there exists a unique $u\in V$ such that $a(u,v)=[f,v]$ for all $v\in V$, where $[\cdot,\cdot]$ represents the dual pairing between $V$ and $V'$.
\end{lemma}
\begin{lemma}[Lemma~1.1, page~87 of \cite{SanchezPalencia1980}]
\label{lem:Korn}\textit{Korn's inequality.}  Given a bounded set $\Upsilon$ with $\partial\Upsilon$ smooth, there exists $\gamma'>0$, such that the following estimate holds: 
\begin{equation*}
\int_\Upsilon E_{ij}(\mathbf{w})\overline{E_{ij}(\mathbf{w})}d\mathbf{x}+\int_{\Upsilon}w^i\overline{w^id}\mathbf{x}\geq \gamma'\norm{\mathbf{w}}_{H^1(\Upsilon)}^2,
\end{equation*}
for all $\mathbf{w}\in H^1(\Upsilon)$.
\end{lemma}  
\section{Useful convergence results.}
\label{CVResult}
\numberwithin{equation}{section}
\numberwithin{theorem}{section}
We list here  the various convergence theorems that are applied throughout the paper. The proofs can be found in \cite{Nguetseng1989}) and \cite{Nguetseng1990}. 
\begin{theorem}\label{2scaleth}
	Let $\mathbf{v_\epsilon}\in L^2(\Omega)$ ($\Omega$ is any bounded open set in $\mathbb{R}^3$) such that: $$\norm{\mathbf{v_\epsilon}}_{L^2(\Omega)}\leq C \hspace{2mm} \text{ for all $\epsilon.$}$$  Then, up to a subsequence (still denoted by $\epsilon$), as $\epsilon\downarrow0$, the following holds: $$\int_\Omega\mathbf{v_\epsilon}\mathbf{w^\epsilon}\phi d\mathbf{x}\rightarrow\int_{\Omega\times Y}\mathbf{v_o}(\mathbf{x},\mathbf{y})\mathbf{w}(\mathbf{y})\phi(\mathbf{x})d\mathbf{x}d\mathbf{y},$$ for all  $\mathbf{w}\in L^2_p$, for all $\phi\in \mathscr{K}(\overline{\Omega})$, where $\mathbf{v_o}\in L^2(\Omega;L^2_p)$.
	
	If, furthermore, $\mathbf{v_\epsilon}\in H^1(\Omega)$ and there exists a constant $C>0$, independent of $\epsilon$, such that:
\begin{equation*}
	\norm{\mathbf{v_\epsilon}}_{L^2(\Omega)}\leq C \hspace{2mm}\text{ for all $\epsilon$,}
\end{equation*}
\begin{equation*}
	\sum_{i=1}^3\int_{\Omega_\epsilon^o}\left|\frac{\partial\mathbf{v_\epsilon}}{\partial x_i}\right|^2d\mathbf{x}\leq C \hspace{2mm}\text{ for all $\epsilon$.}
\end{equation*}
	Then, we can extract a subsequence  (still denoted by $\epsilon$) such that, for all $\mathbf{w}\in L^2_p$ and $\phi\in\mathscr{K}(\overline{\Omega})$, as $\epsilon\downarrow0$, we have: 
\begin{align*}
\mathbf{v_\epsilon}&\rightarrow\mathbf{\langle{v}_o\rangle} \hspace{2mm} \text{in $L^2(\Omega)$-weak,}\\
\int_{\Omega}\mathbf{v_\epsilon}\mathbf{w^\epsilon}\phi d\mathbf{x}&\rightarrow\int_{\Omega\times Y}\mathbf{v_o}(\mathbf{x},\mathbf{y})\mathbf{w}(\mathbf{y})\phi(\mathbf{x})d\mathbf{x}d\mathbf{y},\\
\int_{\Omega_\epsilon^o}\frac{\partial\mathbf{v_\epsilon}}{\partial x_i}\mathbf{w^\epsilon}\phi d\mathbf{x}&\rightarrow\int_{\Omega\times Y_o}\left(\frac{\partial \mathbf{u}}{\partial x_i}(\mathbf{x})+\frac{\partial \mathbf{u_1}}{\partial y_i}(\mathbf{x},\mathbf{y})\right)\mathbf{w}(\mathbf{y})\phi(\mathbf{x})d\mathbf{x}d\mathbf{y},\hspace{0.3cm} i=1,2,3,
\end{align*}
 where $\mathbf{v_o}\in L^2(\Omega,L^2_p)$ is given by: $$\mathbf{v_o}(\mathbf{x},\mathbf{y})=\mathbf{u}(\mathbf{x})+\mathbf{u_r}(\mathbf{x},\mathbf{y})$$ with $\mathbf{u}\in H^1(\Omega)$, $\mathbf{u_r}(\mathbf{x},\mathbf{y})=0$ almost everywhere in $Y_o$, for almost all $\mathbf{x}\in\Omega$; $\mathbf{u_1}\in L^2(\Omega;H_p^1(Y_o)/\mathbb{C})$,  and $\langle\mathbf{v}_o\rangle(\mathbf{x}):=\displaystyle\int_Y\mathbf{v_o}(\mathbf{x},\mathbf{y})d\mathbf{y}$, the mean value of $\mathbf{v_o}(\mathbf{x},\cdot)$.  Moreover, if $\mathbf{v_\epsilon}\in H^1_0(\Omega)$, then $\mathbf{u}\in H^1_0(\Omega)$.
\end{theorem}

\begin{remark}[see \cite{Nguetseng1989}]
\label{th:Nguetseng89}
	Assume that $\norm{\mathbf{v_\epsilon}}_{H^1(\Omega)}\leq C$, for all $\epsilon$.  Then, by extraction of a suitable subsequence, we have: 
	\begin{align*}
	   \mathbf{v_\epsilon}&\rightarrow\mathbf{u}\hspace{4mm}\text{in }\,H^1(\Omega)-weak,\\
	   \int_{\Omega}\frac{\partial \mathbf{v_\epsilon}}{\partial x_i}\mathbf{w^\epsilon}\phi\,d\mathbf{x}&\rightarrow\int_{\Omega\times Y}\left(\frac{\partial\mathbf{u}}{\partial x_i}(\mathbf{x})+\frac{\partial \mathbf{u_1}}{\partial y_j}(\mathbf{x},\mathbf{y})\right)\mathbf{w}(\mathbf{y})\phi(\mathbf{x})\,d\mathbf{x}\,d\mathbf{y},
	\end{align*} 
for all $\mathbf{w}\in \mathbf{L^2_p}$, for all $\phi\in\mathscr{K}(\overline{\Omega})$, where $\mathbf{u_1}\in \mathbf{L^2}(\Omega;\mathbf{H^1_p}(Y)/\mathbb{C})$. 
\end{remark}
These theorems motivate the following definition of two-scale convergence.
\begin{definition}
\label{2cale-conv}
A sequence $\{ v_\varepsilon \}_{\varepsilon>0}$ in $L^2(\Omega)$ is said to \emph{two-scale converge} to $v = v(\mathbf{x},\mathbf{y})$, with $v \in L^2 (\Omega \times Y)$, if and only if:
\begin{equation*}
    \lim_{\varepsilon \to 0} \int_\Omega v_\varepsilon(\mathbf{x}) \psi \left(\mathbf{x}, \frac{\mathbf{x}}{\varepsilon}\right)\, d\mathbf{x} 
    = \frac{1}{|Y|} \int_\Omega \int_Y v(\mathbf{x},\mathbf{y}) \psi(\mathbf{x},\textbf{y})\, d\mathbf{y}\, d\mathbf{x},
\end{equation*}
for any test function $\psi = \psi (\mathbf{x}, \mathbf{y})$, with $\psi \in \mathscr{D}(\Omega, C_p^\infty (Y))$,
see \cite{Nguetseng1989,Allaire1992}.  
\end{definition}

\subsection{Two-scale convergence on the surface}
\label{2scaleboundsec}
To handle the interface term in the weak formulation, we will also need the following theorems, which generalize results about two-scale convergence to sequences in $L^2(\Gamma_\epsilon)$.  Their proofs can be found in \cite{Allaire95}.
\begin{theorem}\cite{Allaire95}
\label{th:2scv-bdary} Let $u_\epsilon$ be a sequence in $L^2(\Gamma_\epsilon)$ such that the surface integral satisfies the bound:
\begin{equation*}
	\epsilon\int_{\Gamma_\epsilon}\left|u_\epsilon(\mathbf{x})\right|^2\,d\sigma_\epsilon(\mathbf{x})\leq C.
\end{equation*}
Then, there exist a subsequence (still denoted by $\epsilon$) and a two-scale limit $g(\mathbf{x},\mathbf{y})\in L^2(\Omega;L^2(\Gamma))$, such that $u_\epsilon(\mathbf{x})$ two-scale converges to $g(\mathbf{x},\mathbf{y})$, in the sense that:
\begin{equation}
	\label{2scvboud}
	\lim_{\epsilon\rightarrow0}\epsilon\int_{\Gamma_\epsilon}u_\epsilon(\mathbf{x})\phi\left(\mathbf{x},\frac{\mathbf{x}}{\epsilon}\right)\,d\sigma_\epsilon(\mathbf{x})=\int_\Omega\int_\Gamma g(\mathbf{x},\mathbf{y})\phi(\mathbf{x},\mathbf{y})\,d\mathbf{x}\,d\sigma(\mathbf{y}),
\end{equation}
for every continuous function $\phi(\mathbf{x},\mathbf{y})\in C[\overline{\Omega};C_{p}(Y)]$.
\end{theorem}
The following extension theorems play a crucial role in establishing the uniform bounds of solutions, which are required for the two-scale convergence.
\subsection{Extension theorems.}
Define: 
\begin{eqnarray}
\Sigma^{s,f}_\epsilon=\partial\Omega\,\cap\,\epsilon\,\tilde{Y}_{s,f}, \hspace{2mm}\mathbf{V}_{s,f}=\left\{\mathbf{v}\in \mathbf{H^1}(\Omega_\epsilon^{s,f})\,:\,\mathbf{v}=0\text{ on }\Sigma^{s,f}_\epsilon \right\}, \label{thm_ext_V}\\
\hspace{2mm}\Omega_1=\left\{\mathbf{x}\in\mathbb{R}^3\,:\,{\rm d}(\mathbf{x},\overline{\Omega})<1\right\}, \label{thm_ext_omega1}
\end{eqnarray}
where $d$ designates the Euclidean metric and $\overline{\Omega}$ is the closure of $\Omega$ in $\mathbb{R}^N$.
\begin{theorem}[Theorem~A of \cite{Nguetseng1990}]
\label{thm:ext1}
For each $\epsilon<\epsilon_o$ ($\epsilon_o$ is a suitable constant), there exists an extension operator $T_\epsilon\in\mathcal{L}(\mathbf{V}_{s},\mathbf{H_0^1}(\Omega_1))$ (i.e., $T_\epsilon$ is continuous linear and $T_\epsilon\mathbf{u}=\mathbf{u}$ on $\Omega_\epsilon^s$, for all $\mathbf{u}\in \mathbf{V_{s}}$) such that: $$\int_{\Omega_1}E_{ij}(T_\epsilon\mathbf{u})\overline{E_{ij}(T_\epsilon\mathbf{u})}\,d\mathbf{x}\leq C\int_{\Omega_\epsilon^s}E_{ij}(\mathbf{u})\overline{E_{ij}(\mathbf{u})}\,d\mathbf{x}\hspace{3mm}\forall \mathbf{u}\in \mathbf{V_\epsilon},$$
where the constant $C$ does not depend on $\epsilon$.
\end{theorem}

A similar extension theorem can be established for $\mathbf{V}_f$.

\begin{theorem}[Theorem~B of \cite{Nguetseng1990}]
\label{thm:ext2}
There exists an extension operator $T_p\in\mathcal{L}(\mathbf{H_p^1}(Y_s),\mathbf{H_p^1})$ such that $T_p\mathbf{w}=\mathbf{w}$ almost everywhere in $Y_s$, for all $\mathbf{w}\in \mathbf{H_p^1}(Y_s)$ and: $$\int_{Y}e_{ij}(T_p\mathbf{w})\overline{e_{ij}(T_p\mathbf{w})}\,d\mathbf{x}\leq C\int_{Y_s}e_{ij}(\mathbf{w})\overline{e_{ij}(\mathbf{w})}\,d\mathbf{x}\hspace{3mm}\forall \mathbf{w}\in \mathbf{H_p^1}(Y_s),$$
where the constant $C$ does not depend on $\epsilon$.
\end{theorem}

A similar extension theorem can be established for $Y_f$.

\end{document}